\documentclass{article}

\usepackage[utf8]{inputenc} 
\usepackage{amsfonts}
\usepackage{amsmath}
\usepackage{amsthm} 
\usepackage{amssymb}
\usepackage{color}
\usepackage{graphicx}
\usepackage{setspace}
\usepackage{enumerate}
\usepackage{booktabs}
\usepackage{subfigure}
\usepackage{pstricks}
\usepackage{url}
\usepackage{hyperref}
\usepackage{mathtools}
\usepackage{enumitem}
\usepackage{cancel}
\usepackage[ruled, boxed, noend, noline, linesnumbered, norelsize]{algorithm2e}
\usepackage[normalem]{ulem}
\usepackage{multicol}
\hypersetup{
	colorlinks,
	linkcolor={red!50!black},
	citecolor={blue!50!black},
	urlcolor={blue!80!black}
}
\newcommand{\doubly}[1]{{\mathcal{D}^{#1}}}	
\newcommand{\distP}{\ell _{\text{poly}}}  
\newcommand{\minFace}{\mathcal{F}_{\min}}

\newcommand{\norm}[1]{\|#1\|}

\newcommand{\reInt}{\mathrm{ri}\,}
\newcommand{\reCone}{\mathrm{rec}\,}
\newcommand{\lineality}{\mathrm{lin}\,}

\newcommand{\closure}{\mathrm{cl}\,}

\newcommand{\spanVec}{\mathrm{span}\,}
\newcommand{\inProd}[2]{\langle #1 , #2 \rangle }

\newcommand{\PSDcone}[1]{{\mathcal{S}^{#1}_+}}

\newcommand{\stdMap}{ {\mathcal{A}}}

\newcommand{\stdCone}{ {\mathcal{K}}}
\newcommand{\stdSpace}{ \mathcal{L}}

\newcommand{\stdFace}{ \mathcal{F}}

\newcommand{\stdInt}{ {e}}

\newcommand{\matRank}{{\mathrm{ rank } \,}}

\newcommand{\T}{\top\hspace{-1pt}}               

\newcommand{\jAlg}{\mathcal{E}}
\renewcommand{\Re}{\mathbb{R}}

\newcommand{\SOC}[2]{{\mathcal{L}^{#2} _{#1}}}
\newcommand{\jProd}[2]{ {#1 \circ #2 } }
\newcommand{\eig}[2]{ \lambda_{#2}^{#1}}

\newcommand{\dist}{ {\mathrm{dist}\,}}
\newcommand{\eigJ}{\lambda}
\newcommand{\comp}{\diamondsuit}

\newcommand{\dpp}{d_{\text{PPS}}}
\newcommand{\ds}{d_{\text{S}}}
\newcommand{\FRF}{facial residual function}
\newcommand{\FRFs}{facial residual functions}
\newcommand{\FRFa}{FRF}
\newcommand{\face}{\mathrel{\unlhd}}
\newcommand{\stdProj}{\mathcal{P}}

\DeclarePairedDelimiter\abs{\lvert}{\rvert}%

\renewcommand{\S}{\mathcal{S}}                    
\newcommand{\tr}{\mathrm{tr}\,}    
\newcommand{\RNum}[1]{\uppercase\expandafter{\romannumeral #1\relax}}
\newtheorem{definition}{Definition}
\newtheorem{lemma}[definition]{Lemma}
\newtheorem{proposition}[definition]{Proposition}
\newtheorem{example}[definition]{Example}
\newtheorem{corollary}[definition]{Corollary}

\newtheorem{theorem}[definition]{Theorem}
\newtheorem*{theorem*}{Theorem}{\bf}{\it}
\newtheorem{assumption}{Assumption}

\newtheorem{remark}[definition]{Remark}

\title{Amenable cones: error bounds without constraint qualifications}

\author{Bruno F. Louren\c{c}o
	\thanks{Department of Mathematical Informatics, Graduate School of Information Science \& Technology, University of Tokyo, 7-3-1 Hongo, Bunkyo-ku, Tokyo 113-8656, Japan.
		(Email: lourenco@mist.i.u-tokyo.ac.jp)}
}
\date{December 2017 (Revised: August 2019)}

\begin{document}

	\maketitle

\begin{abstract}
We provide a framework for obtaining error bounds for linear conic problems without assuming constraint qualifications or regularity conditions.
The key aspects of our approach are the  notions of \emph{amenable cones} and \emph{\FRFs}.  For \emph{amenable cones}, it is shown that error bounds can be expressed as a composition of \FRFs. 
The number of compositions is related to the facial reduction technique and  the singularity degree of the problem. 
In particular, we show that symmetric cones are  amenable and compute {\FRFs}. From that, we are able to furnish a new H\"olderian error bound, thus extending and shedding new light on an earlier result by Sturm on semidefinite matrices. 
We also provide error bounds for the intersection of amenable cones, this will be used to prove error bounds for the doubly nonnegative cone.
At the end, we list some open problems.

	\noindent \textbf{Keywords:} error bounds, amenable cones, facial reduction, singularity degree,	symmetric cones, feasibility problem.

%
\end{abstract}
	
\section{Introduction}
In this work, we are interested in proving error bounds for the following conic feasibility problem.
\begin{align}
\text{find} & \quad x \in (\stdSpace + a) \cap \stdCone \label{eq:feas}\tag{Feas}, 
\end{align}
where $\stdCone$ is a closed convex cone contained in a finite dimensional real vector space $\jAlg$, $\stdSpace \subseteq \jAlg$ is a subspace and $a \in \jAlg$. We will write 
$(\stdCone, \stdSpace, a)$ to denote the problem \eqref{eq:feas}.
We suppose that $\jAlg$ is equipped with some inner product $\inProd{\cdot}{\cdot}$ and that the norm is induced by $\inProd{\cdot}{\cdot}$, i.e., $\norm{x} = \sqrt{\inProd{x}{x}}$.
Given a set $C\subseteq \jAlg$ and $ x \in \jAlg$, we define the distance between $x$ and $C$ as 
$\dist(x,C) = \inf\{\norm{x-y} \mid y\in C\}$.

Suppose that we are given some arbitrary $x \in \jAlg$ and we wish to measure  how far $x$ is from $(\stdSpace + a) \cap \stdCone$. Since $\stdSpace + a$ is an affine space, it is quite simple to compute
$\dist(x,\stdSpace+a)$. Also, in many cases, it is also straightforward to compute  $\dist(x,\stdCone)$. 
Na\"ively, one might expect that if we combine $\dist(x,\stdSpace+a)$ and $\dist(x,\stdCone)$ in some appropriate fashion, we might get a reasonable estimate for $\dist(x, (\stdSpace + a) \cap \stdCone)$. When $\stdCone$ is a polyhedral cone, this is indeed true. In fact, when $\stdCone$ is polyhedral, it follows from the celebrated Hoffman's Lemma that there is a constant $\kappa > 0$ such that 
\begin{equation}\label{eq:hoff}
\dist(x, (\stdSpace + a) \cap \stdCone) \leq \kappa \dist (x,\stdSpace + a) + \kappa \dist(x, \stdCone),
\end{equation}
for every $x \in \jAlg$.  This is an example of an \emph{error bound result}. As far as error bounds go, the polyhedral case is perhaps the best one could hope for. It is \emph{global}, meaning that it holds for all $x \in \jAlg$. No regularity assumptions are needed on the intersection $(\stdSpace + a) \cap \stdCone$.
It is also  \emph{Lipschitzian} meaning that there is a linear relation between the distances, so if we decrease the individual distances to $\stdCone$ and $\stdSpace + a$, the distance to $\stdCone\cap (\stdSpace + a)$ will decrease at least by the same order of magnitude. 

It is  well known that when $\stdCone$ is not polyhedral, the situation can be quite unfavourable and we cannot expect a result as nice as \eqref{eq:hoff} to hold. 
In order to obtain error bounds we need to sacrifice globality, the Lipschitzness or impose regularity conditions.
The literature on error bounds is very rich and it is not possible to do it justice here. Instead, we refer to either the comprehensive survey by Pang \cite{Pang97} or to the chapter by Lewis and Pang \cite{LP98}. 
We emphasize that many results for the nonpolyhedral case include some regularity assumption on the 
intersection $\stdCone \cap (\stdSpace+c)$. 
For instance, compactness and the condition $(\reInt \stdCone )\cap (\stdSpace+c) \neq \emptyset$ (i.e., Slater's condition) might be required for some of the results to hold, see page 313 in \cite{Pang97}.
Also, Baes and Lin recently proved Lipschitzian error bound results for the symmetric cone complementarity problem but they require Slater's condition to hold \cite{BL15}.
For nonlinear semidefinite programs, Yamashita proved error bounds under a few regularity conditions \cite{Y16}.

%

Among the several error bounds results in the literature, the one proved by Sturm in \cite{ST00} is, perhaps, one of the most extraordinary. Here, we provide a brief account. Let 
$\S^n$ denote the space of $n\times n$ symmetric matrices and $\PSDcone{n}$ denote the cone of $n \times n$ symmetric positive semidefinite matrices.  
Given a symmetric matrix $x \in \S^n$, we will denote its minimum eigenvalue by $\lambda _{\min}(x)$.
Combining Theorem 3.3 and Lemma 3.6 of \cite{ST00}, we have the
following result by Sturm.
\begin{theorem*}[Sturm's Error Bound]\label{theo:sturm}
Let $\{x_{\epsilon} \mid 0 < \epsilon \leq 1 \} \subseteq \S^n$ be a \textbf{bounded} set, with the property that $\dist(x_{\epsilon}, \stdSpace + a) \leq \epsilon$ and 
$\lambda _{\min}(x_\epsilon) \geq - \epsilon$, for all $\epsilon \in (0,1]$. Then, there exists 
constants $\kappa > 0$ and  $\gamma \geq 0$ such that 
$$
\dist(x_{\epsilon}, (\stdSpace+a)\cap \PSDcone{n}) \leq \kappa \epsilon^{(2^{-\gamma})},
$$
where $\gamma$ satisfies $\gamma \leq \min \{n-1,
\dim {\stdSpace^\perp \cap \{a\}^\perp},\spanVec(\stdSpace+a) \}$.
\end{theorem*}
There are several remarkable aspects of Sturm's bound.
First of all, no regularity condition is assumed on the 
intersection $\PSDcone{n} \cap (\stdSpace+a)$.
The drawback is that  instead of ``$\epsilon$'', we get ``$\epsilon ^\lambda$'' at the right-hand-side, for some $\lambda \in (0,1]$.
Error bounds of this type are called ``H\"olderian''. 
We emphasize, however, that although the bound is H\"olderian, we know that the exponent is not smaller than $2^{1-n}$. Finally,  Sturm also showed how $\gamma$ can be \emph{computed}, which is 
a significant advancement in comparison to earlier H\"olderian error bounds where it is typically very hard to estimate the exponent, see the comments after Theorems 11 and 13  in \cite{Pang97}.
It turns out that $\gamma$ depends on the \emph{singularity degree} of the system $(\PSDcone{n},\stdSpace,a)$. 
The singularity degree is currently understood as the minimum number of steps that the facial reduction algorithm (by Borwein and Wolkowicz) needs in order to fully regularize $(\PSDcone{n},{\stdSpace},a)$. 
Sturm was also the first to link an error bound result to 
facial reduction. 

The research on facial reduction  \cite{borwein_facial_1981,WM13,pataki_strong_2013,DW17} has shown that problems that do not satisfy Slater's condition are quite numerous.
For those problems, results 
such as Sturm's error bound are useful to derive convergence results.
For a recent application see 
the paper by Drusvyatskiy, Li and Wolkowicz \cite{DGW17}, 
where Sturm's bound plays an important role in deriving a rate of convergence of the alternate projection method for semidefinite feasibility problems that do not satisfy Slater's condition.

Sturm's error bound was later extended to a mixed system of semidefinite  and second order cone constraints, see the chapter by 
Luo and Sturm \cite{sturm_handbook}.
Apart from that, it seems that no other paper attempted to establish further links between error bounds and facial reduction. 
It is not known, for instance, for which 
convex cones a result similar to Sturm's error bound holds.
This paper is, hopefully, a step towards answering 
this question.

\subsection{The contributions of this paper}
Two concepts are introduced in this paper: \emph{amenable cones} and \emph{\FRFs}. 
The main goal is to show that for \emph{amenable cones},  a result analogous to Sturm's error bound holds. 
This article has the following contributions.



\begin{enumerate}
	\item We define amenable cones (Definition \ref{def:beaut}) and prove that polyhedral cones, projectionally exposed cones, symmetric cones and strictly convex cones are amenable (Propositions~\ref{prop:beaut} and \ref{prop:sym_am}). 
	Roughly speaking, a cone $\stdCone$ is amenable if for every face $\stdFace \face \stdCone$, we have that  $\dist(x,\stdCone)$ provides a reasonable upper bound to 
	$\dist(x,\stdFace)$, when $x \in \spanVec \stdFace$.
	
	Furthermore, we observe that amenable cones are nice (Proposition~\ref{prop:am_nice}) and show that 
	amenability is preserved by direct products and by taking injective linear images (Proposition \ref{prop:beaut_prev}).
	\item We define {\FRFs} (Definition~\ref{def:ebtp}). 
	Let  $\stdFace \face \stdCone$ and $z \in \stdFace^*$, where $\stdFace^*$ is the dual cone of $\stdFace$. A {\FRF} provides way of estimating 
	$\dist(x,\stdFace \cap \{z\}^\perp)$ by using other available information such as $\dist(x,\stdCone)$, $\dist(x,\spanVec \stdFace)$ and $\inProd{x}{z}$.
	We prove that symmetric cones admit {\FRFs} of the form 
	$\kappa \epsilon + \kappa \sqrt{\epsilon \norm{x}}$ (Theorem \ref{theo:sym_am}).

	Furthermore, {\FRFs} can be easily constructed for direct products of  amenable cones, provided that {\FRFs} are known for each individual cone.
	Similarly, {\FRFs} are also easily constructed for injective linear images of convex cones. See Proposition \ref{prop:ebtp}.
	\item For \emph{amenable cones}, we prove a novel  error bound 
	result that does not require constraint qualifications. The error bound is expressed as a composition of {\FRFs}. 
	The number of function compositions is connected to facial reduction, see Theorem \ref{theo:err} and Proposition \ref{prop:effic}.
	We then use Theorem \ref{theo:err} to provide two H\"olderian error bounds for symmetric cones, see Theorem \ref{theo:sym_err} and 
	Proposition \ref{prop:sym_err2}. 
	
	We also study error bounds for the intersection of cones and derive a result for the doubly nonnegative cone, see Proposition \ref{prop:doubly}.
\end{enumerate}

This article is divided as follows. In Section \ref{sec:prel} we review  several necessary tools. If the reader already has experience with the material therein, we recommend skipping most of Section \ref{sec:prel}.
In Section \ref{sec:am}, we introduce amenable cones and 
{\FRFs}. In Section \ref{sec:err}, we derive error bound results. The case of symmetric cones is discussed in 
Section \ref{sec:sym}.
In Section \ref{sec:conc}, we summarize this work and point out future 
research directions.

\section{Preliminaries}\label{sec:prel}

\subsection{Basic definitions and assumptions}\label{sec:conv}
We recall our assumption that $\jAlg$ is
equipped with some arbitrary inner product $\inProd{\cdot}{\cdot}$ and that the distance 
function   $\dist(\cdot,\cdot)$ is  computed with respect the  norm $\norm{\cdot}$ induced by $\inProd{\cdot}{\cdot}$.
For a direct product $\jAlg = \jAlg^1\times \jAlg^2$, we will assume that the inner product splits along the product so that $$\inProd{(x_1,x_2)}{(y_1,y_2)} = \inProd{x_1}{y_1} + \inProd{x_2}{y_2},$$ when $(x_1,x_2),(y_1,y_2) \in \jAlg^1 \times \jAlg^2$. By doing so, if $C^1 \subseteq \jAlg^1 $ and $C^2 \subseteq \jAlg^2$ , we have 
\begin{equation}
\dist((x_1,x_2),C^1\times C^2) = \sqrt{\dist(x_1,C^1)^2 + \dist(x_2,C^2)^2 }. \label{eq:prod_proj}
\end{equation}
We remark that because all norms on a finite dimensional vector space are equivalent, our assumption that  the norm is induced by the inner product is not very restrictive. 
 
Let $C \subseteq \jAlg$ be an arbitrary convex set. We will denote its relative interior, closure and linear span by $\reInt C$, $\closure C$ and $\spanVec C$, respectively. 
We will write $C^\perp$ for the orthogonal complement of $C$, which is defined as $$C^\perp = \{x \in \jAlg \mid \inProd{x}{y} = 0, \forall y \in C \}.$$
We recall that a set $\stdCone$ is a convex cone if for all nonnegative $\alpha, \beta$ and all $x, y \in \stdCone$, we have 
$\alpha x + \beta y \in \stdCone$. 
We will write $\stdCone^*$ for the dual 
cone of $\stdCone$ with respect the inner product 
$\inProd{\cdot}{\cdot}$. We have 
$$\stdCone ^* = \{x \in \jAlg \mid \inProd{x}{y} \geq 0, \forall y \in \stdCone \}.$$ 
We write $\lineality \stdCone$ for the lineality space of $\stdCone$, which is defined as $$\lineality \stdCone  = \stdCone \cap - \stdCone.$$
A cone is said to be  \emph{pointed} if $\lineality \stdCone =\{0\}$.

Let $\stdCone$ be a convex cone and $\stdFace \subseteq \stdCone$ be a convex cone contained in $\stdCone$.  $\stdFace$ is a \emph{face} of $\stdCone$  if and only if the property below holds
$$
x,y \in \stdCone, x+y \in \stdFace \Rightarrow x,y \in \stdFace.
$$
In this case, we write $\stdFace \face \stdCone$.
If there exists $z \in \stdCone^*$ such that 
 $\stdFace = \stdCone^* \cap \{z\}^\perp$, then 
$\stdFace$ is said to be an \emph{exposed face}. 
If all the faces of $\stdCone$ are exposed, then 
$\stdCone$ is said to be \emph{facially exposed}.

We define the \emph{conjugate face} of $\stdFace$ with respect to $\stdCone$ as $$\stdFace^{\Delta} = \stdCone^* \cap \stdFace^\perp.$$
Recall that if $\stdFace \face \stdCone$, then
$\stdFace = \stdCone \cap \spanVec \stdFace$. It follows that 
$\stdFace^* = \closure(\stdCone^* + \stdFace^\perp )$. 
A cone $\stdCone$ is said to be \emph{nice} when the closure can be removed, that is, if the following property holds
$$
\stdFace \face \stdCone \Rightarrow \stdCone^* + \stdFace^\perp \text{ is closed}. 
$$
%
Niceness plays an important role in the study of  the facial structure of convex cones.  It is also important in the context of optimality conditions, see, for example, Corollary~4.2 in the work of Borwein and Wolkowicz \cite{BW82}. 
Regularization approaches such as facial reduction have very nice theoretical properties when the underlying cone is nice, see the works by Pataki \cite{pataki_strong_2013,P13}, related works by Tun\c{c}el and Wolkowicz \cite{TW12},
Roshchina \cite{R14} and by Roshchina and Tun\c{c}el~\cite{RT17}. 

In this work, we will need the following technical fact related to niceness.

\begin{proposition}\label{prop:cj_ri}
	Let $\stdCone$ be a closed convex cone such that $\stdCone ^*$ is nice.  
	Let $z \in \stdCone^*$
	and  $\stdFace = \stdCone \cap \{z\}^{\perp}$. Then, 
	$z \in \reInt \stdFace^{\Delta}$.
\end{proposition}
\begin{proof}
	By definition of the conjugate face, we have 
	$z \in \stdFace^{\Delta}$.
	Suppose $z \not \in \reInt \stdFace^{\Delta}$. 
	By invoking a separation theorem (e.g., Theorem 11.3 in \cite{Roc70}), 
	we can find $x \in \stdFace^{\Delta*}$ such that $\inProd{x}{z} = 0$ and 
	$x \not \in \stdFace^{\Delta\perp}$.
	Then, the niceness of $\stdCone^*$ implies that 
	$$\stdFace^{\Delta*} = \stdCone + \stdFace^{\Delta\perp}.$$
	Therefore, $x = u + v$, where $u \in \stdCone$ and 
	$v \in 	\stdFace^{\Delta\perp}$. 
	Since $\inProd{x}{z} = 0$ and $z \in \stdFace^\Delta$, we obtain 
	that $$  \inProd{x}{z} = \inProd{u}{z} = 0,$$ that is, $u \in \stdFace$. Since $\stdFace \subseteq \stdFace^{\Delta \perp}$, we conclude that $x \in \stdFace^{\Delta \perp}$, which is a contradiction.
	\end{proof}
If $\stdMap$ is a linear map, we will denote by 
$\stdMap^\T$ the corresponding adjoint map. 
The operator norm of $\stdMap$ will be denoted by $\norm{A} = \sup\{ \norm{Ax} \mid \norm{x} \leq 1 \}$.	
We will denote the set of nonnegative real numbers by $\Re_+$.
We conclude this subsection with a reminder on our overall assumption on $\stdCone$.
\begin{assumption}\label{asmp:1}
Throughout this paper, we assume that $\stdCone$ denotes a pointed closed convex cone.
\end{assumption}

\subsection{Hoffman's Lemma}
Hoffman's Lemma can be stated in many different ways. 
For the sake of completeness we state below the format we will use throughout this article, which is a consequence
of Hoffman's original result \cite{HF57}. 
We recall that a set $C$ is said to be \emph{polyhedral} 
if it can be expressed as the solution set of a finite system of 
linear inequalities. 

\begin{theorem}[Hoffman's Lemma \cite{HF57}]\label{theo:hoff}
Let $C_1,\ldots,C_m \subseteq \jAlg$ be polyhedral sets such that 
$\cap _{i=1}^m C_i \neq \emptyset$.
There exists a positive constant $\kappa$ such that 
$$
\dist(x, \cap _{i=1}^m C_i) \leq  \kappa\sum _{i=1}^m  \dist(x,C_i),\quad \forall x \in \jAlg.
$$
\end{theorem}

\subsection{Constraint qualifications}\label{sec:cq}
Although we will not assume that $(\stdCone,\stdSpace,a)$ satisfies some constraint qualification, it is still necessary to discuss them. 
We say that $(\stdCone, \stdSpace,a)$ satisfies \emph{Slater's condition} if $(\reInt \stdCone) \cap (\stdSpace+a) \neq \emptyset$.
In this work, however, we will use a weaker constraint qualification called the \emph{partial polyhedral Slater's (PPS) condition}, which is defined as follows.
\begin{definition}[Partial Polyhedral Slater's condition]\label{def:pps}
	Let $\stdCone = \stdCone ^1\times \stdCone ^2$, where  $\stdCone ^1, \stdCone ^2$ are closed convex cones such that $\stdCone ^2$ is polyhedral.
	We say that $(\stdCone,\stdSpace,a)$ satisfies the Partial Polyhedral Slater's (PPS) condition if there exists $(x_1,x_2) \in \stdSpace+a$, such that $x_1 \in \reInt \stdCone^1$ and $x_2 \in \stdCone^2$.
\end{definition}%
The PPS condition reflects the fact that we only care about having a relative interior point with respect the part of the cone that we know that is not polyhedral.
When a conic linear program satisfies the PPS condition, we get the same consequences of the usual Slater's condition: 
zero duality gap and, when the optimal value is finite, the dual problem is attained (e.g., Proposition 23 in \cite{LMT15}).

We will treat Slater's condition as a particular 
case of the PPS condition. 
In fact, if $(\stdCone, \stdSpace, a)$ satisfies the Slater's condition, we can add an extra dummy coordinate, so that 
$(\stdCone \times \{0\}, \stdSpace \times \{0\}, (a,0))$ satisfies the PPS condition. 
Similarly, if $\stdCone$ is a polyhedral cone, we will also consider that 
the PPS conditions holds, since we can also add an extra coordinate 
and take $\stdCone^1 = \{0\}$.

\subsection{Facial Reduction}\label{sec:fra}
The facial reduction algorithm originally appeared in \cite{Borwein1981495} and was developed by Borwein and Wolkowicz  as a way of dealing with conic convex programs that do not satisfy regularity conditions. 
More recently, Pataki \cite{pataki_strong_2013}   and Waki and Muramatsu \cite{WM13} gave simplified descriptions of facial reduction for the special case of conic linear programs.

Suppose that $(\stdCone, \stdSpace, a)$ is feasible.
The basic idea is that  there exists an unique face $\minFace$ of $\stdCone$ with the following properties:
\begin{enumerate}[label=$({\alph*}$)]
	\item $\minFace \cap (\stdSpace+a) = \stdCone \cap (\stdSpace+a)$,
	\item $(\minFace,\stdSpace,a)$ satisfies Slater's condition.
\end{enumerate}
The first property means that the feasible region stays the same when we replace $\stdCone$ by $\minFace$. 
It can be shown that properties $(a)$ and $(b)$ imply that  $\minFace$ is the smallest face of $\stdCone$ containing 
$\stdCone \cap (\stdSpace+a)$, see item $(ii)$ of Proposition 2.2 in \cite{pataki_handbook}. 
For this reason, $\minFace$ is called the \emph{minimal face} of the problem  $(\stdCone, \stdSpace, a)$.


The classical facial reduction algorithm construct a chain of faces as follows:
$$
\minFace =\stdFace _{\ell}  \subsetneq \cdots \subsetneq \stdFace_1 = \stdCone, 
$$
where $\stdFace _{i+1} = \stdFace _{i} \cap \{z_i\}^\perp$ and 
$z_i \in \stdFace_i^* \cap \stdSpace \cap  \{a\}^\perp$, for  $i=1,\ldots, \ell-1$. 
The $z_i$ are  called \emph{reducing directions} and computing them usually forms the bulk of the computational cost of facial reduction.
There are quite a few recent works discussing how to compute those directions and how to do facial reduction efficiently and in a numerical stable manner \cite{csw13,lourenco_muramatsu_tsuchiya3,PP17,Fr16,PFA17,YPT17}. 
We regard finding each $z_i$ as \emph{one facial reduction step}.

\subsubsection{Singularity degree and distance to polyhedrality}\label{sec:sing}
For a fixed $(\stdCone, \stdSpace, a)$, we might need many facial reduction steps before $\minFace$ is reached. 
Motivated by that, we define the \emph{singularity degree} of $(\stdCone, \stdSpace, a)$ as the minimum number of facial reduction steps before 
$\minFace$ is reached.
This definition of singularity degree is adopted, for example, in 
\cite{LP17,DPW15} and in a recent survey~\cite{DW17}.
However, the first usage of \emph{singularity degree} in the context of facial reduction was due to Sturm in \cite{ST00} and it had a slightly different meaning, see section 5.4 and footnote 3 in \cite{LMT15}.
 
In particular, according to Sturm's definition, if $\minFace = \{0\}$, then the singularity degree is zero. This makes perfect sense in the context of \cite{ST00}, since if $\minFace = \{0\}$ then a Lipschitzian error bound holds for $(\stdCone, \stdSpace, a)$, see page 1232 and Equation~(2.5) therein. 
In this paper, we also make a similar observation in Proposition~\ref{prop:err_z}.
Nevertheless, it seems that most researchers are now inclined to define the singularity degree as in \cite{LP17}, so we shall also follow suit.
In this case, if $\minFace = \{0\}$, then the singularity degree should be 
at least one when $\dim{\stdCone} \geq 1$.

We will denote the singularity degree of $(\stdCone, \stdSpace,a)$ by 
$\ds(\stdSpace,a)$. Note that $\ds(\stdSpace,a)$ depends on $\stdCone$, $\stdSpace$ and $a$. However, it is possible to give a bound on 
the singularity degree that does not depend on $\stdSpace$ nor $a$.
In what follows, if we have a chain of faces $\stdFace _{\ell} \subsetneq \cdots \subsetneq \stdFace _{1}$, the 
length of the chain is defined to be $\ell$. 
Then, the \emph{longest chain of faces of $\stdCone$} is denoted by 
$\ell _{\stdCone}$ and is defined as the length of 
the longest chain of face of $\stdCone$ such that all inclusions are strict. We have that $\ds(\stdSpace,a) \leq \ell_{\stdCone}$.

Sometimes it is enough to find a face that satisfies a less strict constraint qualification.
In particular, the \emph{FRA-Poly} algorithm in \cite{LMT15} is divided in two phases. In the first phase, a face satisfying the PPS condition is found and in the second phase, $\minFace$ is computed.
In many cases of interest, this two-phase strategy leads to better bounds on the singularity degree than the classical facial reduction algorithm, see for instance, Table 1 in \cite{LMT15}.
We will recall here a few definitions and results from \cite{LMT15}.

\begin{definition}
	The \emph{distance to polyhedrality} $\distP(\stdCone)$ is the 
	length \emph{minus one} of the longest strictly ascending chain of nonempty faces $\stdFace _{\ell} \subsetneq \cdots \subsetneq \stdFace _{1} $ which satisfies:
	\begin{enumerate}[label=({\it \alph*})]
		\item  $\stdFace _{\ell}$ is polyhedral;  
		\item $\stdFace _j$ is not polyhedral for $j < \ell$.
	\end{enumerate}
\end{definition}
 See Example 1 in \cite{LMT15} for the values of $\distP(\stdCone)$ for some 
 common cones. In particular, if $\stdCone$ is polyhedral, we have 
 $\distP(\stdCone) = 0$. 
 In this paper, we will compute a bound for $\distP(\stdCone)$ when $\stdCone$ is a symmetric cone, see Remark~\ref{rem:sym}.
The next result gives an upper bound to the number of facial reduction 
steps that are necessary before a face satisfying the PPS condition is found. 
\begin{proposition}\label{prop:fra_poly}
Let $\stdCone = \stdCone^1\times \cdots \times \stdCone^s$, where 
each $\stdCone^i$ is a pointed closed convex cone.
Suppose $(\stdCone,\stdSpace,a)$ is feasible.
There is a chain of faces  
$$
\stdFace _{\ell}  \subsetneq \cdots \subsetneq \stdFace_1 = \stdCone 
$$
 of length $\ell$ and vectors $(z_1,\ldots, z_{\ell-1})$ satisfying the following properties.
\begin{enumerate}[label=({\it \roman*})]
	\item $\ell -1\leq  \sum _{i=1}^{s} \distP(\stdCone ^i)  \leq \dim{\stdCone}$
	\item For all $i \in \{1,\ldots, \ell -1\}$, we have
	\begin{flalign*}%
	z_i &\in \stdFace _i^* \cap \stdSpace^\perp \cap \{a\}^\perp, &\\
		\stdFace _{i+1} &= \stdFace _{i} \cap \{z_i\}^\perp.&
	\end{flalign*}		
	\item $\stdFace _{\ell} \cap (\stdSpace+a) = \stdCone \cap (\stdSpace + a)$ and 	$(\stdFace _{\ell},\stdSpace,a)$ satisfies the PPS condition.
\end{enumerate}
\end{proposition}
\begin{proof}
As mentioned previously, \emph{FRA-Poly} is divided in two phases \cite{LMT15}. 
In Phase~1, it computes  the directions $z_i$ as in 
item $(ii)$. Then, it ends with a face satisfying the PPS condition, as 
in item $(iii)$.
The bound on the number of directions follows from item $(i)$ of Proposition 8 in \cite{LMT15} and from the fact that $\distP(\stdCone ^i) \leq \dim{\stdCone^i}$ for every $i$.\footnote{Note that if $\stdFace \face \stdCone$ and 
	$\stdFace \subsetneq \stdCone$, them $\dim \stdFace < \dim \stdCone $.} 	
\end{proof}
%
In this paper, we define the quantity $\dpp(\stdSpace,a)$, which 
is the minimum number of reduction directions needed to find a 
face $\stdFace \face \stdCone$ that contains $\stdCone \cap (\stdSpace+a)$ and such that 
$(\stdFace, \stdSpace,a)$ satisfies the PPS condition.

\subsection{Distance functions and generalized eigenvalue functions  }\label{sec:eig}
In this subsection, we will briefly discuss a generalization of the concept of eigenvalues introduced by Renegar in \cite{RE16}.
Let $\stdCone$ be a pointed closed convex cone and $d \in \reInt \stdCone$, then the generalized 
eigenvalue function of $\stdCone$ with respect to $d$ is 
\begin{equation}
\eig{d}{\stdCone}(x) = \inf \{t \mid x - td \not \in \stdCone \}.\label{eq:dist_def}
\end{equation}
With that, we have $x - \eig{d}{\stdCone}(x)d \in \stdCone$  for all 
$x \in \spanVec \stdCone$. 
We also have 
\begin{align*}
x \in \stdCone &\iff \eig{d}{\stdCone}(x) \geq 0\\
x \in \reInt \stdCone &\iff \eig{d}{\stdCone}(x) > 0.
\end{align*}
We observe that if $\stdCone = \PSDcone{n}$ and $d$ is the $n\times n$ identity matrix, then 
$\lambda _{\min}(x) = \eig{d}{\stdCone}(x) $, for all $x \in \jAlg$.
Renegar proved the following result in \cite{RE16}, see Proposition 2.1 therein.
\begin{proposition}
	Let $\stdCone$ be a closed pointed convex cone and $d \in \reInt \stdCone$. Then, the function $\eig{d}{\stdCone}(x)$ is concave and Lipschitz continuous over $\spanVec \stdCone$.
\end{proposition}
The Lipschitz continuity of $\eig{d}{\stdCone}(\cdot)$ is important because it implies that it is reasonable to use $\eig{d}{\stdCone}(x)$ as an indirect way of measuring $\dist(x,\stdCone)$.  This idea is expressed in the next proposition.

\begin{proposition}\label{prop:sur}
	Let $d \in \reInt \stdCone$. There are positive constants $\kappa _1$ and $\kappa _2$ such that 
	$$
	\kappa _1  \max(-\eig{d}{\stdCone}(x),0) \leq \dist(x,\stdCone) \leq \kappa _2 \max(-\eig{d}{\stdCone}(x),0), \quad \forall x \in  \spanVec \stdCone. 
	$$
\end{proposition} 
\begin{proof}	
	If $x \in \stdCone$, then we have $\eig{d}{\stdCone}(x) \geq 0$ and $ \dist(x,\stdCone) = 0$, so we are done. 
	Suppose that $x\not \in \stdCone$.
	Since 	$x - \eig{d}{\stdCone}(x)d \in \stdCone$, we have 
	$$
	\dist(x,\stdCone) \leq \norm{x - (x- \eig{d}{\stdCone}(x)d)}= -\eig{d}{\stdCone}(x)\norm{d }. 
	$$
	Let $v \in \stdCone$ be such that $\dist(x,\stdCone) = \norm{x-v}$. Since $x \not \in \stdCone$, $v$ belongs to the relative boundary of $\stdCone$, so that $$\eig{d}{\stdCone}(v) = 0, \quad \eig{d}{\stdCone}(x)<0.$$  Using the Lipschitz continuity of $\eig{d}{\stdCone}(\cdot)$, there is some $\tilde \kappa $ such that 
	$$
	-\eig{d}{\stdCone}(x) = \abs{\eig{d}{\stdCone}(x) - \eig{d}{\stdCone}(v)} \leq \norm{x-v}\tilde \kappa = \dist(x,\stdCone) \tilde \kappa, 
	$$
	for all $x \in \spanVec \stdCone$.
	We conclude that the proposition holds with $\kappa _1= 1/{\tilde{\kappa}}$ and $\kappa _2 = \norm{d}$.
\end{proof}

\section{Amenable cones and {\FRFs}}\label{sec:am}
In this section, we introduce the two notions that are the cornerstones of this work: \emph{amenable cones} and \emph{\FRFs}.%
\subsection{Amenable cones}
\begin{definition}\label{def:beaut}
	A closed convex cone $\stdCone$ is said to be \emph{amenable} 
	 if for every 
	face $\stdFace \face \stdCone$ there is a positive constant $\kappa$ such that 
\begin{equation}\label{eq:def_am}
	\dist (x,\stdFace) \leq \kappa \dist (x, \stdCone),\qquad \forall x \in \spanVec \stdFace.
\end{equation}
\end{definition}

Next, we prove that a few common cones are amenable. 
The proof that symmetric cones are amenable will be deferred to Proposition \ref{prop:sym_am}.
We recall that a pointed cone $\stdCone$ is said to be \emph{strictly convex}
if the only faces besides $\stdCone$ and $\{0\}$ are extreme rays 
(i.e., one dimensional faces). Also, $\stdCone$ is said to be \emph{projectionally exposed}
if for every face $\stdFace \face \stdCone$ there is a projection (not necessarily orthogonal) $\stdProj$ 
such that $\stdProj(\stdCone) = \stdFace$, see \cite{ST90} by Sung and Tam.
Here, we remind that a projection is a linear map $\stdProj:\jAlg \to \jAlg$ satisfying $P^2 = P$. 
If for every face $\stdFace \face \stdCone$ there is an \emph{orthogonal} projection $\stdProj$ such that $\stdProj(\stdCone) = \stdFace$, then $\stdCone$ is said to be \emph{orthogonal projectionally exposed}. 
\begin{proposition}\label{prop:beaut}
	The following cones are amenable.
	\begin{enumerate}[label=$(\roman*)$]
		\item Projectionally exposed cones. In particular, if $\stdFace \face \stdCone$ and $\stdProj$ is a projection satisfying $\stdProj(\stdCone) = \stdFace$, then \eqref{eq:def_am} is satisfied with $\kappa = \norm{\stdProj}$. 
		\item Polyhedral cones.
		\item Strictly convex cones.
	\end{enumerate}	
\end{proposition}
\begin{proof}
	\begin{enumerate}[label=$(\roman*)$]
		\item Let $\stdFace \face \stdCone$ and $\stdProj:\jAlg\to \jAlg$ be a projection 
		map such that $\stdProj(\stdCone) = \stdFace$. Then, we have 
		$\stdProj(\spanVec \stdCone) = \spanVec \stdFace$. 
		Furthermore, since 	$\stdProj^2 = \stdProj$, we have $P(\spanVec \stdFace) = \spanVec \stdFace$.
		
		Now, let $x \in \spanVec \stdFace$ and let $y \in \stdCone$ be such that 
		$\dist(x, \stdCone) = \norm{x-y}$. Then, since $\stdProj(y) \in \stdFace$ and 
		$\stdProj(x) = x$, we have
		$$
		\dist(x,\stdFace) \leq \norm{x-\stdProj(y)} = \norm{\stdProj(x) - \stdProj(y)} \leq 
		\norm{\stdProj}\norm{x-y},
		$$
		where $\norm{\stdProj}$ is the operator norm of $\stdProj$.  
		This shows that 		\eqref{eq:def_am} is satisfied with $\kappa = \norm{\stdProj}$.
		\item Let $\stdCone$ be a polyhedral cone and 
		$\stdFace$ be a face of $\stdCone$. 
		Since $\stdCone$ is polyhedral, $\stdFace$ must be an exposed face (e.g., Corollary 2 in \cite{TM76}), therefore there exists $z \in \stdCone^*$ such that 
		$$
		\stdFace = \stdCone \cap \{z\}^\perp.
		$$
		By Hoffman's Lemma (Theorem \ref{theo:hoff}), there is a 
		positive constant $\kappa$ such that 
		$$
		\dist(x, \stdFace) \leq \kappa \dist(x, \stdCone) + \kappa \dist(x, \{z\}^\perp).		
		$$

		We now observe that if $x \in \spanVec \stdFace$ then $x \in \{z\}^\perp$. Therefore,
		\begin{align*}
		\dist(x,\stdFace) &\leq \kappa\dist(x,\stdCone),\quad \forall x \in \spanVec \stdFace.
		\end{align*}
		
		\item Let $\stdCone$ be a strictly convex cone and let $\stdFace$ be a proper face of $\stdCone$.
		If $\stdFace = \{0\}$, since $\spanVec \stdFace = \{0\}$, it is enough to take $\kappa = 1$. 
		
		We move on to the case where $\stdFace = \{\alpha v \mid \alpha \geq 0\}$, for some nonzero $v \in \stdCone$. We  assume, without loss of generality, that $\norm{v} = 1$. 
		Then $\spanVec \stdFace = \{\alpha v \mid \alpha \in \Re \}$. Note that if $\alpha \geq 0$, we have
		$$
		\dist(\alpha v, \stdFace) = \dist(\alpha v, \stdCone) = 0.
		$$
		So suppose that $\alpha < 0$. Let $u \in \stdCone$ be such that 
		$\dist(-v,\stdCone) = \norm{u+v} $. We have
		\begin{align*}
		\dist(\alpha v, \stdFace)    = -\alpha, \qquad 		\dist(\alpha v, \stdCone)  = -\alpha \norm{u+v}. 
		\end{align*}
		It follows that for every $x \in \spanVec \stdFace$, we have
		$$
		\dist(x, \stdFace)   \leq \frac{1}{\norm{u+v}}	\dist(x,\stdCone).
		$$
		We remark that, since $\stdCone$ is pointed, $-v \not \in \stdCone$, so $\norm{u+v} > 0$.
	\end{enumerate}
\end{proof}
\begin{remark}\label{rem:const}
In Definition~\ref{def:beaut}, the constant $\kappa$ may depend on $\stdFace$. 
Nevertheless, there are 	cones that admit a finite ``universal'' constant $\kappa_{\stdCone}$ depending only on $\stdCone$ and such that \eqref{eq:def_am} holds for all faces. For example, we will see in Proposition~{\ref{prop:sym_am}} that $\kappa_{\stdCone} = 1$ is enough for symmetric cones. Also, if $\stdCone$ is polyhedral, since the number of faces is finite,  we may pick a constant $\kappa _{\stdFace}$ for each face and let   $\kappa_{\stdCone}$ be the maximum among the $\kappa _{\stdFace}$. Finally, if $\stdCone$ is a pointed strictly convex cone, the proof of item~$(iii)$ of Proposition~\ref{prop:beaut} shows that we may take
$$
\kappa _{\stdCone} = \sup _{v \in \stdCone, \norm{v} = 1} \dist(-v,\stdCone)^{-1} $$
as a universal constant for $\stdCone$. Since $\dist(\cdot, \stdCone)$ is a continuous function, 
$ \kappa _{\stdCone}$ must be finite.

However, if $\stdCone$ is a projectionally exposed cone, the proof of item $(i)$ only shows that 
the constant $\kappa$ can be taken to be the operator norm of the projection $\mathcal{P}$.  Because 
 $\mathcal{P}$ is not necessarily orthogonal, $\norm{\mathcal{P}}$ can be quite large and we do not known whether a finite universal constant exists.
More generally, it is not known whether arbitrary amenable cones admit finite universal constants.  
\end{remark}
We note that amenability is preserved by simple operations, see Appendix~\ref{app:proof} for proofs.
\begin{proposition}[Preservation of amenability]\label{prop:beaut_prev}
The following hold.
\begin{enumerate}[label=$(\roman*)$]
	\item If $\stdCone^1$ and $\stdCone^2$ are amenable cones then 
	$\stdCone^1 \times \stdCone^2$ is amenable.
	\item  If $\stdCone$ is amenable and $\stdMap$ is an injective linear map, then $\stdMap(\stdCone)$ is amenable.
\end{enumerate}	
\end{proposition}
Next, we  examine the connection between amenability and related concepts.
Two sets $S_1,S_2$ are said to have \emph{subtransversal intersection at $\overline{x} \in S_1 \cap S_2$} if there exists a positive $\kappa$ and a neighbourhood $V$ of $ \overline{x}$ such that 
\begin{equation}
\dist(x, S_1\cap S_2) \leq \kappa(\dist(x,S_1)+\dist(x,S_2)), \quad \forall x \in V.\label{eq:subt}
\end{equation} 
Subtransversality is discussed extensively in Ioffe's book \cite{Io17}, see Chapter~7 and Definition 7.5 therein. 
Another related concept is \emph{bounded linear regularity}.  The sets $S_1,S_2$ are said to be boundedly linearly regular, if $S_1 \cap S_2 \neq \emptyset$ and for every bounded set $B\subseteq \jAlg$ there exists $\kappa _B$ such that 
$$\dist(x,S_1\cap S_2) \leq \kappa _{B}\max(\dist(x,S_1),\dist(x,S_1)), \quad \forall x \in B.$$
See, for example, the work by Bauschke, Borwein and Li \cite{BBL99}. Now, recalling that $\stdFace = \stdCone \cap \spanVec \stdFace$ holds for every face $\stdFace \face \stdCone$, we obtain the following proposition, see Appendix~\ref{app:proof} for the proof.
\begin{proposition}\label{prop:am_others}
Let $\stdCone$ be a closed convex cone and $\stdFace \face \stdCone$. The following are equivalent.
\begin{enumerate}[label=$(\roman*)$]
	\item There exists a positive constant $\kappa$ such that $\dist (x,\stdFace) \leq \kappa \dist (x, \stdCone)$ holds for every $x \in \spanVec \stdFace$.
	\item $\stdCone$ and $\spanVec \stdFace$ intersect at $0$ subtransversally.
	\item $\stdCone$ and $\spanVec \stdFace$ are boundedly linearly regular. 
\end{enumerate}	
\end{proposition}

The overall conclusion is that	$\stdCone$ is amenable if and only if every face $\stdFace \face \stdCone$ is such that $\stdCone$ and $\spanVec \stdFace$ are boundedly linearly regular or intersect subtransversally at the origin.


This is good news because, thanks to Theorem~10 in \cite{BBL99}, it turns out that  two (not necessarily pointed) convex cones $\stdCone_1,\stdCone_2$ are boundedly linearly regular if and only if 
$-\stdCone_1^* -\stdCone_2^*$ is closed and there is $\alpha > 0$ such that 
\begin{equation}\label{eq:prop_g}
U \cap (-\stdCone_1^* -\stdCone_2^*) \subseteq \alpha( (-\stdCone_1^*\cap U) + (-\stdCone_2^*\cap U) ),
\end{equation}
where $U = \{x \in \jAlg \mid \norm{x} \leq 1\}$ is the unit ball. This leads us to the following  result.
\begin{proposition}\label{prop:am_nice}
Amenable cones are nice and, in particular, are facially exposed.
\end{proposition}
\begin{proof}
Let $\stdFace \face \stdCone$ be a face of an amenable cone $\stdCone$.
By the preceding discussion, $\stdCone$ and $\spanVec \stdFace$ are boundedly linearly regular, so Theorem~10 in \cite{BBL99} implies that $-\stdCone^* + \stdFace^{\perp}$ is closed. Therefore, 
$\stdCone^* + \stdFace^{\perp}$ is closed and $\stdCone$ must be nice.
To conclude, we recall that Pataki proved in Theorem~3 of \cite{P13} that all nice cones are facially exposed. 
\end{proof}
\begin{remark}
Propositions~\ref{prop:beaut} and \ref{prop:am_nice} together imply that projectionally exposed cones are nice. This has been proved earlier by  Permenter \cite{Per16}.

We do not know whether nice cones must necessarily be amenable. At this moment, this seems unlikely because it would imply that the condition given in \eqref{eq:prop_g} (also called \emph{property $(G)$} in \cite{BBL99}) is somehow superfluous when $\stdCone_1=\stdCone$, $\stdCone_2 = \spanVec \stdFace$ and $\stdFace \face \stdCone$. Nevertheless, a nice but not amenable cone remains to be found.
\end{remark}

Recently, Roshchina and Tun\c{c}el introduced 
the concept of \emph{tangentionally exposed cones} in \cite{RT17} and they
showed that nice cones are always tangentially exposed, although the converse does not hold in general, see Example~2 in \cite{RT17}. 
As amenable cones are nice, they must be tangentially exposed as well.

We conclude this section by showing how amenability might break down.
\begin{example}[A non-amenable cone, Figure~\ref{fig:non_bt}]
	Let $C \subseteq \Re^2$ be the smallest closed convex set containing
	$$
	\{(x,x^2) \in \Re^2 \mid 0 \leq x\leq 1 \} \cup \{(x,0) \in \Re^2 \mid -1 \leq x \leq 0 \}.
	$$	
	Let $\stdCone \subseteq \Re^3$ be the smallest closed convex cone containing $C\times \{1\}$. 
	As seen in Figure~\ref{fig:non_bt}, $\stdCone$ is not facially exposed, so $\stdCone$ cannot be amenable, because of Proposition~\ref{prop:am_nice}.
	Nevertheless, we will check precisely why the amenability condition fails.

	Let $\hat C = \{(x,0)  \mid -1 \leq x \leq 0 \}$ and $\stdFace$ be the smallest closed convex cone containing 
	$\hat C \times \{1\}$.	 Since $\hat C$ is a face of $C$, $\stdFace$ is  a face of $\stdCone$. We have
	\begin{align*}
	\stdFace &= \{(x,0,z) \in \Re^3 \mid 0 \leq -x \leq z \}
	\end{align*}		
	Now, let $x \in (0,1]$. We consider the
	point $(x,0,1) \in \spanVec \stdFace$. 
	The projection of $(x,0,1)$ on $\stdFace$ is $(0,0,1)$. Therefore,
	$\dist((x,0,1), \stdFace) = x$. 
	However, 
	$$\dist((x,0,1),\stdCone) \leq \norm{(x,0,1) - (x,x^2,1) }   = x^2.$$
	Therefore, the quotient $\dist((x,0,1),\stdFace)/ \dist((x,0,1),\stdCone)$ gets unbounded as $x$ goes to zero, thus showing that Definition \ref{def:beaut} can never be satisfied for any positive constant $\kappa$. 
	
	\begin{figure}
		\centering\includegraphics[scale=0.6]{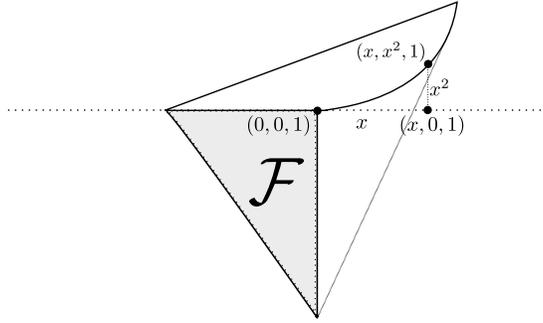}\caption{A cone that is not amenable.}\label{fig:non_bt}
	\end{figure}
	
\end{example}

\subsection{Facial residual functions}
Let $\stdFace$ be a face of $\stdCone$, $z \in \stdFace^*$ and $\hat \stdFace = \stdFace \cap \{z\}^\perp$.
The motivation for 
the definition of \emph{\FRFs} comes from the fact that if for some $x$ we have
$$
\dist(x,\stdCone) = \inProd{x}{z} = \dist(x,\spanVec \stdFace) = 0
$$
then it must be the case that $x \in \hat \stdFace $. This is because for any face $\stdFace \face \stdCone$ we have $\stdFace = \stdCone \cap \spanVec \stdFace$.
If $x$ almost satisfies the equations above, we would hope that the distance between $x$ and $\hat \stdFace$ would also be small. Unfortunately, that is not what happens in general and we usually have to take into account the norm of $x$. 
Accordingly, we settle for the less ambitious goal that  $\dist(x,\stdCone)$ should be bounded by some function $\psi_{\stdFace,z}$ that also depends on the norm of $x$. However, this dependency is not completely arbitrary and we require $\psi_{\stdFace,z}$ to be zero if $x$ belongs to $\hat \stdFace$.

\begin{definition}[Facial residual functions]\label{def:ebtp}
	Let $\stdCone$ be a closed convex cone and $\stdFace$ a face of $\stdCone$.
	Let $z \in \stdFace^*$ and $\hat \stdFace = \stdFace \cap \{z\}^{\perp}$.
	Suppose that $\psi_{\stdFace,z} : \Re _+ \times \Re _+ \to \Re_+$ satisfies the following properties:
	\begin{enumerate}[label=({\it \roman*})]
		\item $\psi_{\stdFace,z}$ is nonnegative, monotone nondecreasing in each argument and $\psi(0,\alpha) = 0$ for every $\alpha \in \Re_+$.
		\item whenever $x \in \spanVec \stdCone$ satisfies  the inequalities
		$$
		\dist(x,\stdCone) \leq \epsilon, \quad \inProd{x}{z} \leq \epsilon, \quad \dist(x, \spanVec \stdFace ) \leq \epsilon
		$$
		we have:
		$$
		\dist(x,  \hat \stdFace)  \leq 
		\psi_{\stdFace,z} (\epsilon, \norm{x}).
		$$	
	\end{enumerate}
	Then, $\psi_{\stdFace,z}$ is said to be \emph{{\FRF} (\FRFa) for $\stdFace$ and $z$}. 
\end{definition}

It not obvious whether {\FRFs} always exist, so will now 
take a look at this issue.  
 Let $\bar \psi _{\stdFace,z}(\epsilon,\norm{x})$ be the optimal value of the following problem.
\begin{align}
\underset{v \in \spanVec \stdCone}{\sup} & \quad \dist(v,  \hat \stdFace) \label{eq:can}\tag{P}\\ 
\mbox{subject to} & \quad \dist(v,\stdCone) \leq \epsilon \nonumber \\ 
&\quad \dist(v, \spanVec \stdFace ) \leq \epsilon \nonumber\\
&\quad \inProd{v}{z} \leq \epsilon \nonumber\\
&\quad \norm{v} \leq \norm{x}\nonumber
\end{align}
The functions $\dist(\cdot,\stdCone)$ and $\dist(\cdot,\spanVec \stdFace) $ are continuous convex functions. 
Since $x$ is fixed in \eqref{eq:can}, the feasible region of \eqref{eq:can} is a compact convex set, due to the presence of the constraint ``$\norm{v} \leq \norm{x}$''.
In particular, $\bar \psi _{\stdFace,z}(\epsilon,\norm{x})$ is finite and nonnegative. 
Furthermore, increasing either $\epsilon$ or $\norm{x}$ enlarges the feasible region, so that  $\bar \psi _{\stdFace,z} (\cdot,\cdot)$ is monotone nondecreasing in each argument. 
If $\epsilon = 0$ and $v$ is feasible for \eqref{eq:can} it must be the case 
that $v \in \hat \stdFace$, so $\dist(v, \hat \stdFace) = 0$. 
Therefore, $\epsilon = 0$ implies $\bar \psi _{\stdFace,z} (0,\alpha) = 0$ for every $\alpha \in \Re_+$.
This shows that $\bar \psi _{\stdFace,z} (\cdot,\cdot)$ is indeed 
a {\FRF} and we will call $\bar \psi _{\stdFace,z}$ the \emph{canonical {\FRF} for $\stdFace$ and $z$}.
It is the best possible, since, by definition,  
$\bar \psi _{\stdFace,z}(\epsilon,\norm{x}) \leq \psi _{\stdFace,z}(\epsilon,\norm{x})$, if 
$\psi _{\stdFace,z}$ is another \FRF.

The existence of  canonical {\FRF} shows that, in principle, error  bounds for amenable cones can always be established, see Theorem \ref{theo:err}.
Unfortunately, computing $\bar \psi _{\stdFace,z}$ is  complicated since it boils down to \emph{maximization} of a convex function over a convex set.
It is also likely that $\bar \psi _{\stdFace,z}$ will have no easy formula as a function of 
$\epsilon$ and $\norm{x}$.

In face of these difficulties, one of the goals in this paper is to show that many useful cones admit simpler {\FRFs}. 
For example, we will show in Theorem \ref{theo:sym_am} that for symmetric cones, 
we can use $\kappa \epsilon + \kappa \sqrt{\epsilon \norm{x}}$ as a {\FRF}, where $\kappa$ is a positive constant.

We say that 
a function $\tilde \psi_{\stdFace,z}$ is a \emph{positive rescaling of $\psi_{\stdFace,z}$} if there are positive constants $M_1,M_2,M_3$ such that 
$$\tilde \psi _{\stdFace,z}(\epsilon,\norm{x}) = M_3\psi_{\stdFace,z} (M_1\epsilon,M_2\norm{x}).$$
Two functions $\psi_1, \psi_2$ are the same \emph{up to positive rescaling} if $\psi _1$ is equal to a positive rescaling of $\psi _2$.
It is possible that $\psi_{\stdFace,z}$ is  different for each choice of $\stdFace$ and $z$. However, in a few cases of interest such as symmetric cones,  $\psi_{\stdFace,z}$ can be taken to be a positive rescaling of 
the same fixed \FRF, see Theorem \ref{theo:sym_am}.

In the next proposition, we will see that when we perform a simple operation on a cone $\stdCone$, we may still use the same facial residual functions for $\stdCone$ if we positive rescale them. As the proof is long but routine, it is deferred to Appendix~\ref{app:proof}.

\begin{proposition}\label{prop:ebtp}
	The following hold.
	\begin{enumerate}[label=$(\roman*)$]
		\item Let $\stdCone = \stdCone^1 \times \stdCone^2$, where 
		$\stdCone^1 \subseteq \jAlg^1,\stdCone^2 \subseteq \jAlg^2$ are amenable cones.

		Let 
		$\stdFace \face \stdCone$ and $z \in \stdFace^*$. Write 
		$\stdFace = \stdFace^1\times \stdFace^2$ where $\stdFace_1 \face \stdCone^1$, $\stdFace _2\face \stdCone^2$. 
		Write $z = (z_1,z_2)$ with $z_1 \in (\stdFace^1)^*$ and $z_2 \in (\stdFace^2)^*$.

		Let $\psi_{\stdFace_1,z_1}$, $\psi_{\stdFace_2,z_2}$ be  {\FRFs} for $\stdFace_1,z_1$ and $\stdFace_2,z_2$, respectively.
	
		Then, there is 	 a positive rescaling of $\psi_{\stdFace_1,z_1}+\psi_{\stdFace_2,z_2}$
		that is also a {\FRF} for $\stdFace,z$.	
		
		\item Let $\stdMap$ be an injective linear map. 
		
		Let $\stdMap(\stdFace) \face \stdMap(\stdCone)$, where 
		$\stdFace \face \stdCone$. Let $z \in (\stdMap(\stdFace))^*$.
		
		Let $\psi _{\stdFace,\stdMap^{\T}z}$ be a {\FRF} for 
		${\stdFace,\stdMap^{\T}z}$.
		
		Then, there is a positive rescaling of $\psi _{\stdFace,\stdMap^{\T}z}$
		that is a {\FRF} for $\stdMap(\stdFace),z$.
	\end{enumerate} 	
\end{proposition}

We will now show that polyhedral cones admit {\FRFs} that 
are linear in $\epsilon$ and do not depend on $\norm{x}$.

\begin{proposition}\label{prop:poly_amenable}
	Let $\stdCone$ be a polyhedral cone and $\stdFace$ a face of $\stdCone$.
	Let $z \in \stdFace^*$ and 
	$\hat \stdFace = \stdFace \cap \{z\}^{\perp}$.
	Then, there is a positive  constant $\kappa$ (depending on $\stdCone,\stdFace,z$)  such that 
	whenever $x$ satisfies  the inequalities
	$$
	\dist(x,\stdCone) \leq \epsilon, \quad \inProd{x}{z} \leq \epsilon, \quad \dist(x, \spanVec \stdFace ) \leq \epsilon
	$$
	we have:
	$$
	\dist(x,\hat \stdFace)  \leq \kappa \epsilon.
	$$	
	
	That is, we can take $\psi _{ \stdFace,z}(\epsilon,\norm{x}) = \kappa \epsilon$ as a {\FRF} for $ \stdFace$ and $z$.
\end{proposition}
\begin{proof}
Suppose $x$ satisfies $\dist(x,\stdCone) \leq \epsilon, \inProd{x}{z} \leq \epsilon$ and $\dist(x, \spanVec \stdFace ) \leq \epsilon$.  
The face $\stdFace$ can be written as the nonempty intersection of two polyhedral sets $$\stdFace = \stdCone \cap \spanVec \stdFace.$$ 
Therefore, from Hoffman's Lemma (Theorem \ref{theo:hoff}), there exists $\kappa_1$ (not depending on $x$) such that 
$$
\dist(x, \stdFace) \leq \kappa_1 (\dist(x,\stdCone) + \dist(x,\spanVec \stdFace)) \leq 2\epsilon \kappa_1. 
$$
Therefore, there exists $v$ such that $\norm{v} \leq 2\epsilon\kappa_1$ such that $x+v \in \stdFace$. Since $\inProd{x}{z} \leq \epsilon$ and $\inProd{x+v}{z} \geq 0$, we obtain
$$
-2\epsilon \kappa _1\norm{z} \leq \inProd{x}{z} \leq \epsilon(1+2\epsilon \kappa _1\norm{z} ). 
$$
Therefore, 
\begin{equation}
\abs{\inProd{x}{z}} \leq \epsilon(1+2 \kappa _1\norm{z} )\label{eq:poly:am}
\end{equation}	
The face $\hat \stdFace$ can be written as the nonempty intersection of 
three polyhedral sets $$\hat \stdFace = \stdCone \cap \spanVec \stdFace \cap \{z\}^\perp.$$
 From Hoffman's Lemma, there is $\kappa_2 > 0 $ such that 
 $$
 \dist(x, \hat \stdFace) \leq \kappa_2(\dist(x,\stdCone) + \dist(x,\spanVec \stdFace) + \dist(x,\{z\}^\perp)). 
 $$
Note that  $\dist(x,\{z\}^\perp)) = \abs{\inProd{x}{z}}/\norm{z}$. From \eqref{eq:poly:am}, we obtain
$$
\dist(x, \hat \stdFace) \leq \epsilon\kappa _2\left(2+\frac{1+2\kappa _1\norm{z}}{\norm{z}}\right).
$$	
We then take $\kappa = \kappa _2(2+\frac{1+2 \kappa _1\norm{z}}{\norm{z}}) $ to conclude the proof.
\end{proof}
Proposition \ref{prop:poly_amenable} is not useful by itself, since we can readily obtain error bounds directly from Hoffman's Lemma. 
However, there are cases where we have to deal with the direct product of polyhedral cones and nonpolyhedral cones. 
Then, since we can take as {\FRFa}s the sum of the individual {\FRFa}s (item $(i)$ of Proposition \ref{prop:ebtp}), it becomes clear that the polyhedral cones only give linear contributions to the overall sum.
This means that all source of non-Lipschitzness and 
nastiness in the error bounds must come from the nonpolyhedral parts, 
which is unsurprising but serves as a sanity check for the theory developed here.


\section{Error bounds}\label{sec:err}
We recall that our goal is to obtain error bounds for $(\stdCone,\stdSpace,a)$ without assuming regularity conditions. Namely, given some arbitrary $x$ we would like to bound 
$\dist(x,\stdCone\cap(\stdSpace+a))$ by some quantity involving $\dist(x,\stdCone)$ and 
$\dist(x,\stdSpace+a)$. 

Our first result is an error bound that is useful in situations where, for some reason, we know a face $\stdFace$ of $\stdCone$ that contains the feasible region of $(\stdCone, \stdSpace,a)$ and such that $(\stdFace, \stdSpace, a)$ satisfies the PPS condition. 
In particular, this covers the case where we know $\minFace$, which is the minimal face of $\stdCone$ that contains $\stdCone \cap (\stdSpace + a)$.

\begin{proposition}[Error bound for when a face satisfying the PPS condition is known]\label{prop:err2}
	Let $\stdCone$ be a closed convex amenable cone
	and let $\stdFace$ denote a face of $\stdCone$ containing $(\stdSpace+a) \cap \stdCone$ and  such that 	the PPS condition is satisfied.
	
	Then, there is a positive constant $\kappa$ (depending on $\stdCone,\stdSpace,a,\stdFace$) such that whenever  $x\in \spanVec \stdCone$ and $\epsilon$ satisfy the inequalities
	$$
	\quad \dist(x,\stdCone) \leq \epsilon, \quad \dist(x,\stdSpace + a) \leq \epsilon, \quad \dist(x, \spanVec \stdFace) \leq \epsilon,
	$$
	we have 
	$$
	\dist\left(x, (\stdSpace + a) \cap \stdCone\right) \leq \kappa \norm{x}\epsilon + \kappa \epsilon.
	$$
	
\end{proposition}
\begin{proof}	
	Since the PPS condition is satisfied for $(\stdFace,\stdSpace,a)$, at least one of the statements below must be true (see Section \ref{sec:cq}). 
	\begin{enumerate}
		\item $\stdFace$ is polyhedral.
		\item $(\reInt \stdFace) \cap (\stdSpace+a)\neq \emptyset $.
		\item $\stdFace = \stdFace^1 \times \stdFace^2$ where $\stdFace^1$ and $\stdFace^2$ are closed convex cones such that $\stdFace^2$ is polyhedral and 
		$$
		((\reInt \stdFace^1)\times \stdFace^2 )\cap (\stdSpace + a) \neq \emptyset.
		$$
	\end{enumerate}
	Recall that cases 1. and 2. can be seen as special cases 
	of 3. if we add extra dummy coordinates.
	Therefore, without loss of generality, we assume that 
	$3.$ holds.
	
	Due to the amenability of $\stdCone$, there is $\kappa _1$ such that
	\begin{equation}
	\dist(z,\stdFace) \leq \kappa_1\dist(z,\stdCone), \quad \forall z \in \spanVec \stdFace. \label{eq:amen1}
	\end{equation}
	Now, let $u$ be such that $\norm{u}\leq \epsilon$ and 
	$x + u \in \spanVec \stdFace$. We have 
	$$\dist(x+u,\stdFace) \leq \kappa_1\dist(x+u,\stdCone) \leq 2\kappa_1\epsilon.$$ 
	Then, observing that 
	$\dist(x,\stdFace) \leq \dist(-u,\stdFace) + \dist(x+u,\stdFace)$, we obtain that $$\dist(x,\stdFace)\leq (1+2\kappa_1)\epsilon.$$	
	Next, since 
	$\dist(x,\spanVec \stdFace) \leq \epsilon$ and $\stdFace \subseteq (\spanVec \stdFace^1)\times \stdFace^2 $, we conclude that
	$$
	\dist(x,(\spanVec\stdFace^1)\times \stdFace^2) \leq \dist(x,\stdFace) \leq (1+2\kappa _1)\epsilon.
	$$	 
	Let $\hat \kappa _1 = (1+2\kappa_1)$.
	Since $\stdFace^2$ is a polyhedral cone and $(\stdSpace + a)\cap ( (\spanVec \stdFace^1)\times \stdFace^2) \neq \emptyset$, we can invoke Hoffman's Lemma (Theorem \ref{theo:hoff})
	which tells us that  there exists a constant $\kappa _2$ such that whenever $x \in \spanVec \stdCone$ satisfies 
	$$
	\dist(x,\stdSpace + a) \leq \epsilon,  \quad \dist(x,(\spanVec \stdFace^1)\times \stdFace^2) \leq \hat \kappa_1 \epsilon	
	$$
	we have $$\dist(x,(\stdSpace + a)\cap ( (\spanVec \stdFace^1)\times \stdFace^2)) \leq  \epsilon \kappa_2.$$
	Therefore, there is $y$ such that $\norm{y} \leq \epsilon \kappa _2$ and
	$$x + y \in (\stdSpace + a)\cap( (\spanVec \stdFace^1)\times \stdFace^2).
	$$ 
	Since $x+ y \in (\spanVec \stdFace^1)\times \stdFace^2$, we can write $x+y = (z_1,z_2)$, with $z_1 \in \spanVec \stdFace^1$ and 
	$z _2 \in \stdFace _2$.
	By \eqref{eq:amen1} and since $x + y $ lies in $\spanVec \stdFace$, we have  
	\begin{equation}
	\dist(z_1,\stdFace^1) \leq \dist (x + y,\stdFace) \leq \kappa _1 \dist (x + y,\stdCone)	\leq \epsilon(\kappa_1 + \kappa_1\kappa_2).\label{eq:bt}
	\end{equation}
	Since the PPS condition is satisfied, there exists $$d = (d_1,d_2) \in ((\reInt \stdFace^1)\times \stdFace^2)\cap (\stdSpace+a).$$ 
	By Proposition \ref{prop:sur}, there is 
	$\kappa _3 > 0$ such that 
	\begin{equation}
	-\eig{d_1}{\stdFace^1}(z_1) \kappa _3\leq \dist(z_1,\stdFace^1).  \label{eq:be_er2}
	\end{equation} 
	Let $t_{\epsilon} = \epsilon (\kappa_1 + \kappa_1\kappa_2)/\kappa _{3}$. It follows from \eqref{eq:bt} and \eqref{eq:be_er2} that 
	$z_1 + t_{\epsilon} d_1 \in \stdFace^1$. As $d_2,z_2 \in \stdFace^2$, we conclude that
	\begin{align}
	x + y + t_{\epsilon}d \in \stdFace. \label{eq:aux:er}
	\end{align}
	We have
	\begin{align*}
	x + y + t_{\epsilon}d = (x+y-a) + t_{\epsilon}(d-a)  + a(1+t_{\epsilon}).
	\end{align*}
	Furthermore, since $x+y \in \stdSpace+a$ and $d \in \stdSpace + a$, the first two terms of the 
	right hand side belong to $\stdSpace$. 
	Therefore, if we divide the whole expression by $(1+t_{\epsilon})$ we get
	\begin{align*}
	\frac{x+y + t_{\epsilon} d}{1+t_{\epsilon}} \in (\stdSpace + a)\cap \stdFace.
	\end{align*}
	We conclude that
	\begin{align*}
	\dist(x,(\stdSpace + a)\cap \stdFace) & \leq \norm{x-\left(\frac{x+y + t_{\epsilon} d}{1+t_{\epsilon}}\right) } \\
	& \leq  \norm{x}\frac{t_{\epsilon}}{1+t_{\epsilon}} + \frac{\epsilon\kappa_2}{1+t_{\epsilon}} + \norm{d}\frac{t_{\epsilon}}{1+t_{\epsilon}}\\
	& \leq  \norm{x}t_{\epsilon} + \epsilon\kappa_2 +  \norm{d}t_{\epsilon}\\
	& \leq  \kappa \norm{x}\epsilon + \kappa \epsilon,
	\end{align*}
	where $\kappa = \max\{(\kappa_1 + \kappa_1\kappa_2)/\kappa _{3},\kappa_2+\norm{d}(\kappa_1 + \kappa_1\kappa_2)/\kappa _{3}\}$.
\end{proof}
Proposition \ref{prop:err2} has the following immediate corollary, where  $\dist(x,\spanVec \stdFace)$ is embedded directly into the error bound.
\begin{corollary}\label{col:beauty}
	Let $\stdCone$ be a closed convex amenable cone and 
	$\stdFace \face \stdCone$ a face of $\stdCone$ containing $(\stdSpace+a) \cap \stdCone$ and  such that 	$(\stdFace,\stdSpace,a)$ satisfies the PPS condition.
	
	Then, there is a positive constant $\kappa$ (depending on $\stdCone,\stdSpace,a,\stdFace$) such that whenever  $x \in \spanVec \stdCone$ and $\epsilon$ satisfy the inequalities
	$$
	\quad \dist(x,\stdCone) \leq \epsilon, \quad \dist(x,\stdSpace + a) \leq \epsilon
	$$
	we have 
	$$
	\dist\left(x, (\stdSpace + a) \cap \stdCone\right) \leq (\kappa \norm{x}+ \kappa )(\epsilon+\dist(x, \spanVec \stdFace)).
	$$	
\end{corollary}
\begin{proof}
	We apply the previous proposition by taking $\hat \epsilon =  \dist(x,\stdCone) + \dist(x,\stdSpace + a) + \dist(x, \spanVec \stdFace)$, which tells us that 
	$$
	\dist\left(x, (\stdSpace + a) \cap \stdCone\right) \leq  (\kappa \norm{x} + \kappa)(\dist(x,\stdCone) + \dist(x,\stdSpace + a) + \dist(x, \spanVec \stdFace)).
	$$
	Adjusting the constant $\kappa$, we get that whenever $x \in \spanVec \stdCone$ satisfies
	$$
	\quad \dist(x,\stdCone) \leq \epsilon, \quad \dist(x,\stdSpace + a) \leq \epsilon
	$$
	we have
	$$
	\dist\left(x, (\stdSpace + a) \cap \stdCone\right) \leq (\kappa \norm{x}+ \kappa )(\epsilon+\dist(x, \spanVec \stdFace)).
	$$	
\end{proof}

From Proposition \ref{prop:err2} and Corollary \ref{col:beauty}, it 
becomes clear that the key to general error bounds 
for $(\stdCone, \stdSpace, a)$ is to know some face $\hat \stdFace$ of $\stdCone$ for which the PPS condition is satisfied \emph{and} we should also know some bound on $\dist(x,\spanVec \hat \stdFace)$. 

This is where we will use facial reduction (Section \ref{sec:fra}). If $(\stdCone,\stdSpace,a)$ is feasible, but 
the PPS condition is not satisfied, then there exists 
$z_1 \in \stdCone^* \cap \stdSpace \cap \{a\}^\perp$ with 
$z_1 \not \in \stdCone^\perp$, e.g., Theorem 4 in \cite{LMT15}.
In particular, $\stdFace_1 \coloneqq \stdCone \cap \{z_1\}^\perp$ is a proper face of
$\stdCone$ that contains the feasible region of 
$(\stdCone, \stdSpace,a)$. 
Again, if $(\stdCone \cap \{z_1\}^\perp, \stdSpace,a)$ still does not 
satisfy the PPS condition, we use the same principle to obtain a new $z_2$ together with the face $\stdFace _2 \coloneqq \stdCone \cap \{z_1\}^\perp\cap \{z_2\}^\perp$.
Then, we proceed until a face satisfying the PPS condition is found. 
At each step,  we will use a {\FRF} to keep track of 
the distance between $x$ and the face $\stdFace _i$.
The next proposition is the first step towards this idea.



\begin{proposition}\label{prop:res_feas}
	Let $(\stdCone, \stdSpace,a)$ be feasible.
	Let $\stdFace$ be a face of $\stdCone$, 
	\begin{align*}
	z \in \stdFace^* \cap \stdSpace^{\perp}\cap\{a\}^\perp\\
	\hat \stdFace = \stdFace \cap \{z\}^{\perp},
	\end{align*}
	 with $z \neq 0$.
	 Let $\psi _{\stdFace,z}$ be a {\FRF} for $\stdFace$ and $z$. Then,  there is a positive rescaling of  $\psi _{\stdFace,z}$ such that whenever $x \in \spanVec \stdCone$ satisfies  the inequalities
	$$
	\dist(x,\stdCone) \leq \epsilon, \quad \dist(x,\stdSpace + a) \leq \epsilon
	$$
	we have:
	\begin{align*}
	\dist(x, \hat \stdFace)  \leq 
	\psi _{\stdFace,z} (\epsilon + \dist(x, \spanVec \stdFace ), \norm{x}). 
	\end{align*}
\end{proposition}
\begin{proof}
	Positive rescaling $\psi_{\stdFace,z}$ if necessary, we may assume that $\psi_{\stdFace,z}$ is such that whenever $x \in \spanVec \stdCone$ and $\tilde \epsilon$ satisfy  the inequalities
	$$
	\dist(x,\stdCone) \leq 2\tilde \epsilon, \quad \inProd{x}{z} \leq 2\tilde \epsilon\norm{z}, \quad \dist(x, \spanVec \stdFace ) \leq \tilde \epsilon
	$$
	we have
	$$
	\dist(x, \hat \stdFace)  \leq 
	\psi_{\stdFace,z} (\tilde \epsilon, \norm{x}).
	$$
	Let $$\tilde \epsilon = \frac{\dist(x,\stdCone)}{2} + \frac{\abs{\inProd{x}{z}}}{2\norm{z}} + \dist(x, \spanVec \stdFace ).$$ 
	Then, the following inequality holds for every $x \in \spanVec \stdCone$.  
	\begin{equation}
	\dist(x, \hat \stdFace)  \leq \psi_{\stdFace,z} \left(\frac{\dist(x,\stdCone)}{2} + \frac{\abs{\inProd{x}{z}}}{2\norm{z}} + \dist(x, \spanVec \stdFace ), \norm{x}\right).	\label{eq:fr1}
	\end{equation}
	Now, suppose that $x \in \spanVec \stdCone$ satisfies
	$$
	\dist(x,\stdCone) \leq \epsilon, \quad \dist(x,\stdSpace + a) \leq \epsilon.
	$$	
	Since $\dist(x,\stdSpace + a) \leq \epsilon$, there exists $u$ such that $\norm{u} \leq \epsilon$ and $x + u \in \stdSpace + a$. Because $z$ is orthogonal to 
	$\stdSpace + a$, it follows that $\inProd{x+u}{z} = 0$ and that	
	\begin{equation}
	\abs{\inProd{x}{z}}  \leq \norm{z}\epsilon. \label{eq:z1}
	\end{equation}
	Finally, from \eqref{eq:fr1}, \eqref{eq:z1}, $\dist(x,\stdCone) \leq \epsilon$ and the monotonicity of $\psi_{\stdFace,z}$, we obtain that 
	$\dist(x,  \hat \stdFace)  \leq 
	\psi_{\stdFace,z} (\epsilon + \dist(x,\spanVec \stdFace), \norm{x})$.
\end{proof}

For what follows, we introduce a special notation for function composition. 
Let $f:\Re\times \Re \to \Re$ and $g:\Re\times \Re \to \Re$ be real functions.
We define $f\comp g$ to be the function satisfying 
$$
(f\comp g)(a,b) = f(a+g(a,b),b),
$$
for every $a,b \in \Re$. Note that if $f$ and $g$ are monotone nondecreasing on each argument, then the same is true for $f\comp g$. 
	
\begin{lemma}\label{lem:am_chain}
Let $\stdSpace\subseteq \jAlg$ be a subspace and $a \in \jAlg$.
Let 
$$
\stdFace _{\ell}  \subsetneq \cdots \subsetneq \stdFace_1 = \stdCone 
$$
be a chain of faces of $\stdCone$ together with $z_i \in \stdFace _i^*\cap\stdSpace^\perp \cap \{a\}^\perp$ such that 
$\stdFace_{i+1} = \stdFace _i\cap \{z_i\}^\perp$, for $i = 1,\ldots, \ell - 1$. 
For those $i$, let $\psi _{i}$ be a {\FRF} for $\stdFace_{i}, z_{i}$.
Then, there is a positive rescaling of the $\psi_i$ such that if $x \in \spanVec \stdCone$ satisfies the inequalities
	$$
	\dist(x,\stdCone) \leq \epsilon, \quad \dist(x,\stdSpace + a) \leq \epsilon
	$$
	we have:
	$$
	\dist(x,  \stdFace _{\ell})  \leq 
	\varphi (\epsilon,\norm{x}),
	$$
where $\varphi = \psi _{{\ell-1}}\comp \cdots \comp \psi_{{1}}$, if $\ell \geq 2$. If $\ell = 1$, we let $\varphi$ be the function satisfying $\varphi(\epsilon, \norm{x}) = \epsilon$.
\end{lemma}

\begin{proof}
The case $\ell = 1$ is straightforward. 
For the case $\ell \geq 2$, we proceed by induction. 	
When $\ell = 2$, we apply Proposition \ref{prop:res_feas} to 
$\stdCone, \stdFace_1, z_1$ and $\psi _1$. 
Therefore, after positive rescaling $\psi_1$ appropriately, whenever $x \in \spanVec \stdCone$ satisfies the inequalities
$$
\dist(x,\stdCone) \leq \epsilon, \quad \dist(x,\stdSpace + a) \leq \epsilon
$$
we have:
$$
\dist(x, \stdFace _{2})  \leq 
\psi_{{1}} (\epsilon + \dist(x,\spanVec\stdFace_1),\norm{x}).
$$
In this case, since $x \in \spanVec \stdCone$ and $\stdFace _1 = \stdCone$, we 
have $\dist(x,\spanVec \stdFace_1) = 0$.

We now suppose that the lemma holds for chains of length $\hat \ell$ and will show 
that it must hold when the length is $\hat \ell+ 1$. By the inductive hypothesis,
we have that whenever $$
\dist(x,\stdCone) \leq \epsilon, \quad \dist(x,\stdSpace + a) \leq \epsilon
$$
we have:
$$
\dist(x,  \stdFace _{\hat \ell})  \leq 
(\psi _{{\hat \ell}-1}\comp \cdots \comp \psi_{{1}} )(\epsilon,\norm{x}).
$$
From the the definition of $\psi _{{\hat \ell} }	$ and its monotonicity in the first argument we get
\begin{align*}
\dist(x, \stdFace _{\hat \ell+1 })  & \leq \psi _{{\hat \ell} }(\epsilon + \dist(x, \spanVec \stdFace _{\hat \ell}),\norm{x}  ) \\
& \leq  \psi _{{\hat \ell} }(\epsilon +(\psi _{{\hat \ell-1}}\comp \cdots \comp \psi_{{1}} )(\epsilon,\norm{x}),\norm{x} ) \\
& \leq 	(\psi _{{\hat \ell}}\comp \cdots \comp \psi_{{1}} )(\epsilon,\norm{x}),
\end{align*}
where we used the fact that  $\dist(x, \spanVec \stdFace _{\hat \ell}) \leq \dist(x, \stdFace _{\hat \ell})  $ to obtain the second inequality.
\end{proof}
Using Lemma \ref{lem:am_chain}, we obtain one of the main results of this paper.
\begin{theorem}[Error bound for amenable cones]\label{theo:err}
Let $\stdSpace\subseteq \jAlg$ be a subspace and $a \in \jAlg$.
Let $\stdCone$ be a closed convex \emph{amenable cone} and let 
$$
\stdFace _{\ell}  \subsetneq \cdots \subsetneq \stdFace_1 = \stdCone 
$$
be a chain of faces of $\stdCone$  together with $z_i \in \stdFace _i^*\cap\stdSpace^\perp \cap \{a\}^\perp$ such that 
 $(\stdFace _{\ell}, \stdSpace,a)$ satisfies the 
 PPS condition and $\stdFace_{i+1} = \stdFace _i\cap \{z_i\}^\perp$ for every $i$. 
For $i = 1,\ldots, \ell - 1$, let $\psi _{i}$ be a {\FRF} for $\stdFace_{i}$, $z_i$. 
Then, after positive rescaling the $\psi _{i}$, there is a  positive constant $\kappa$ (depending on $\stdCone, \stdSpace, a, \stdFace _{\ell}$) such that if  $x \in \spanVec \stdCone$ satisfies the inequalities
$$
\quad \dist(x,\stdCone) \leq \epsilon, \quad \dist(x,\stdSpace + a) \leq \epsilon,
$$
	we have 
	$$
	\dist\left(x, (\stdSpace + a) \cap \stdCone\right) \leq (\kappa \norm{x} + \kappa )(\epsilon+\varphi(\epsilon,\norm{x})),
	$$
where $
\varphi = \psi _{{\ell-1}}\comp \ldots \comp \psi_{{1}}$, if $\ell \geq 2$. If $\ell = 1$, we let $\varphi$ be the function satisfying $\varphi(\epsilon, \norm{x}) = \epsilon$. 
\end{theorem}
\begin{proof}
The case $\ell = 1$ follows from Proposition \ref{prop:err2}, by 
taking $\stdFace = \stdFace_1$.
Now, suppose $\ell \geq 2$.
We apply Lemma \ref{lem:am_chain}, which tells us that, after positive rescaling the $\psi _i$,  
if $x \in \spanVec \stdCone$ satisfies 	
	$$
	\dist(x,\stdCone) \leq \epsilon, \quad \dist(x,\stdSpace + a) \leq \epsilon
	$$
	we have:
	$$
	\dist(x, \stdFace _{\ell})  \leq 
	\varphi(\epsilon,\norm{x}),
	$$
where $\varphi = \psi _{{\ell-1}}\comp \ldots \comp \psi_{{1}} $.
Since $\stdCone$ is amenable and $(\stdFace _{\ell}, \stdSpace,a)$ satisfies the 
PPS condition, we invoke Corollary \ref{col:beauty} which implies that 
	$$
	\dist\left(x, (\stdSpace + a) \cap \stdCone\right) \leq (\kappa \norm{x} + \kappa )(\epsilon+\varphi(\epsilon,\norm{x})),
	$$
for a positive constant $\kappa$ depending on $\stdCone, \stdSpace, a, \stdFace _{\ell}$.
\end{proof}

We now clarify a few aspects of Theorem \ref{theo:err}. 
First of all, Theorem \ref{theo:err} assumes that there is a chain of faces ending in a face $\stdFace _{\ell}$ such that $(\stdFace _{\ell}, \stdCone, a)$ satisfies the PPS condition. 
The existence of such a chain is a nontrivial consequence of facial reduction theory.
In particular, its existence follows from Proposition \ref{prop:fra_poly}.
It also follows from Theorem 3.2 in \cite{WM13} or from Theorem 1 in \cite{pataki_strong_2013}.

Now, that the question of existence of a chain satisfying the requirements of Theorem \ref{theo:err} is settled, we will take a look at efficiency issues. If we fix $(\stdCone, \stdSpace,a)$ there could be several chains of faces that meet the criteria in Theorem \ref{theo:err}. 
Since it is desirable to have an error bound with $\ell$ as small as possible, we will use facial reduction theory to give bounds on $\ell$. Here,  we recall that $\dpp(\stdSpace,a)$ is the minimal number of reducing directions needed to find a face that satisfies the PPS condition and 
$\ds(\stdSpace,a)$ is the singularity degree, see Section \ref{sec:sing}.

\begin{proposition}[Efficiency of the error bound]\label{prop:effic}
Let $\stdCone = \stdCone^1\times \ldots \times \stdCone^s$, where 
each $\stdCone^i$ is a closed convex cone.
Suppose $(\stdCone, \stdSpace, a)$ is feasible.
Then there is a chain of faces of $\stdCone$ 
$$
\stdFace_{\dpp(\stdSpace,a)+1} \subsetneq \cdots \subsetneq  \stdFace _1  = \stdCone  
$$
satisfying the requirements of Theorem \ref{theo:err} such that the following bounds are satisfied
\begin{enumerate}[label=$(\roman*)$]
\item $\dpp(\stdSpace,a) \leq  \sum _{i=1}^s \distP(\stdCone^i)$
\item $\dpp(\stdSpace,a) \leq \dim(\stdSpace^\perp\cap \{a\}^\perp) $
\item $\dpp (\stdSpace,a) \leq \ds(\stdSpace,a) $.
\end{enumerate}
\end{proposition}
\begin{proof}
By definition, there exists at least one chain of length 
$\dpp(\stdSpace,a)+1$ satisfying the requirements of Theorem \ref{theo:err}.
The bound in item~$(i)$ follows from Proposition~\ref{prop:fra_poly}.
We will now prove item $(ii)$. Let 
\begin{equation}\label{eq:pro_eff:1}
\stdFace_{\dpp(\stdSpace,a)+1} \subsetneq \cdots \subsetneq  \stdFace _1  = \stdCone  
\end{equation}
be a chain of faces of $\stdCone$  together with $z_i \in \stdFace _i^*\cap\stdSpace^\perp \cap \{a\}^\perp$ such that 
 $\stdFace_{i+1} = \stdFace _i\cap \{z_i\}^\perp$ for every $i$. 
The inclusions in \eqref{eq:pro_eff:1} must be strict, otherwise we would be able to remove 
some faces of the chain, shrink it and contradict the minimality of $\dpp(\stdSpace,a)$. 
Finally, we note that for $i > 1$,  if $z_{i}$ belongs to the space 
spanned by $\{z_{1},\ldots, z_{i-1} \}$, then we would have $\stdFace_{i+1} = \stdFace_{i}$. Therefore, $\{z_1,\ldots, z_{\dpp(\stdSpace,a) } \}$ is a linear independent set contained in $\stdSpace^\perp\cap \{a\}^\perp$.

Item $(iii)$ holds because the PPS condition is less strict than 
Slater's condition, so a chain of faces ending with a face for which 
Slater's condition holds will also satisfy the requirements of Theorem \ref{theo:err}.
\end{proof}
In particular, Proposition \ref{prop:effic} shows that the number of function compositions appearing in Theorem \ref{theo:err} can be taken 
to be no more than the singularity degree of $(\stdCone, \stdSpace,a)$.

\begin{remark}\label{rem:dist}
	Let $d \in \reInt \stdCone$ and consider the 
	generalized eigenvalue function $\eig{d}{\stdCone}(\cdot)$ defined in 
	Section \ref{sec:eig}. 
	From Proposition \ref{prop:sur}, there is a 
	constant $\kappa' > 0$ depending on $d$ such that 
	$$
	\eig{d}{\stdCone}(x) \geq -\epsilon \quad \Rightarrow \quad \dist(x, \stdCone) \leq \kappa'\epsilon,
	$$
	for all $x \in \spanVec \stdCone$.
	Therefore, under the setting of Theorem \ref{theo:err}, we get that the inequalities
	$$
	\eig{d}{\stdCone}(x) \geq -\epsilon, \quad \dist(x,\stdSpace + a) \leq \epsilon
	$$
	imply
	$$
	\dist\left(x, (\stdSpace + a) \cap \stdCone\right) \leq (\kappa \norm{x} + \kappa )((\kappa' + 1)\epsilon+\varphi((\kappa' + 1)\epsilon,\norm{x})),
	$$
	where $\kappa$ is some positive constant.
	Noting that $\varphi((\kappa' + 1)\epsilon,\norm{x})$ is a positive rescaling 
	of $\varphi(\epsilon, \norm{x})$, we see that Theorem \ref{theo:err} is still valid if we replace ``$\dist(x, \stdCone) \leq \epsilon$'' by 
	``$\eig{d}{\stdCone}(\cdot) \geq -\epsilon$''.

Similarly, if $\stdSpace+a$ is described as the solution set of 
some system of linear equalities ``$\stdMap x = b$'', we can substitute ``$\dist(x, \stdCone) \leq \epsilon$'' by some quantity measuring the error with 
respect that system. For instance, we could use 
``$\sum _{i=1}^m |b_i - \stdMap _i (x)| \leq \epsilon$'', where the $\stdMap _i$ are such that 
$\stdMap(x) = (\stdMap _1 (x), \ldots, \stdMap _{m}(x))$.
	
\end{remark}
Next, we will make a brief detour and generalize an observation made by Sturm in \cite{ST00}.
He noticed that if $(\PSDcone{n},\stdSpace,a)$ is such that $\minFace = \{0\}$, then a Lipschitzian error bound holds, see (2.5) in \cite{ST00}.
First, we need the following auxiliary result.

\begin{lemma}\label{lemma:reint_dual}
	Let $z \in \reInt \stdCone^*$. 
	Then, there is a positive constant  $\kappa$ such that 
	$$
	\norm{x} \leq \kappa \inProd{x}{z}, \qquad \forall x \in \stdCone.
	$$
\end{lemma}
\begin{proof}
	Let $C = \{x \in \stdCone \mid \inProd{x}{z} = 1  \}$.
	The recession cone of $C$ is the set $$\reCone C = \{x \in \stdCone \mid \inProd{x}{z} = 0 \}.$$
	If $x\in \stdCone$, $x \not \in ({\stdCone^{*}})^\perp$ and $\inProd{x}{z} = 0$, then $\{x\}^\perp$ is a hyperplane that properly separates $z$ from $\stdCone^*$. Such a hyperplane exists if and only if $z \not \in \reInt \stdCone^*$, see Theorem 20.2 in \cite{Roc70}.
	We conclude that $\reCone C \subseteq (\stdCone^*)^\perp$.
	
	Since $\lineality \stdCone = (\stdCone^*)^\perp$ and $\stdCone$ is pointed (Assumption~\ref{asmp:1}), we have  $\reCone C = \{0\}$. 
	Therefore, $C$ must be compact. Let $\kappa = \sup _{u \in C}\norm{u}$. Then for nonzero $x \in \stdCone$, we have 
	$$
	\frac{\norm{x}}{\inProd{x}{z}} \leq \kappa. 
	$$
\end{proof}

\begin{proposition}[Error bound for trivial intersections]\label{prop:err_z}
Suppose that $(\stdCone,\stdSpace,a)$
is such that
$$(\stdSpace+a)\cap \stdCone =\{0\}.$$	 Then, there exists a positive constant $\kappa$ (depending on $\stdCone,\stdSpace,a$) such that 
	$$
	\dist(x,\stdCone) \leq \epsilon,\quad \dist(x,\stdSpace + a) \leq \epsilon \quad \Rightarrow \quad \norm{x} \leq \kappa \epsilon.
	$$
\end{proposition}
\begin{proof}
Since $(\stdSpace+a)\cap \stdCone =\{0\}$ holds, we have, in particular, 
$0 \in 	\stdSpace+a$. Therefore, $\stdSpace+a = \stdSpace$.
We conclude that $\stdSpace\cap \stdCone = \{0\}$. 
By the Gordan-Stiemke's Theorem (see Corollary 2 in
Luo, Sturm and Zhang \cite{Luo97dualityresults}), $\stdSpace\cap \stdCone = \{0\}$ holds if and only if there exists $z \in (\reInt \stdCone^*) \cap \stdSpace^\perp$. 

Since $\dist(x,\stdSpace) \leq \epsilon$, there exists $u$ such that $\norm{u} \leq \epsilon$ and $x+u \in \stdSpace$. 
Since $\inProd{x+u}{z} = 0$, we conclude that 
	\begin{equation}
	\inProd{x}{z} \leq \epsilon\norm{z}.\label{eq:prop:err_z1:1}
	\end{equation}
	Since $\dist(x,\stdCone) \leq \epsilon$, there exists 
	$v$ such that $\norm{v} \leq \epsilon$ and $x+v\in \stdCone$. 
	By Lemma \ref{lemma:reint_dual} and \eqref{eq:prop:err_z1:1}, there exists a positive
	constant $\kappa _1$ such that
	\begin{equation}
	\norm{x} - \norm{v}  \leq \norm{x+v} \leq \kappa _1\inProd{x+v}{z} \leq 2\kappa _1\norm{z}\epsilon. \label{eq:prop:err_z1:2}
	\end{equation} 
	From \eqref{eq:prop:err_z1:2}, we conclude that the proposition holds with $\kappa = \epsilon(1 + 2\kappa_1 \norm{z})$. 
\end{proof}

\subsection{Error bounds for symmetric cones}\label{sec:sym}
In this subsection, we use Theorem \ref{theo:err} to prove error bounds for symmetric cones. First, we need to review a few aspects of Jordan algebras.
More 
details can be found in the books by Koecher \cite{K99}, Faraut and Kor\'anyi \cite{FK94} and also 
in the survey article by Faybusovich \cite{FB08}.
A \emph{Euclidean Jordan algebra} is a finite dimensional real vector space $\jAlg$ equipped with a bilinear product $\jProd{}{}:\jAlg\times \jAlg \to \jAlg$ (the Jordan product) and an inner product $\inProd{\cdot}{\cdot}$ satisfying the following axioms: 
\begin{enumerate}[label=$(\arabic*)$]
	\item $\jProd{x}{y} = \jProd{y}{x}$,
	\item $\jProd{x}({\jProd{x^2}{y}}) = \jProd{x^2}({\jProd{x}{y}})$, where $x^2 = \jProd{x}{x}$,
	\item $\inProd{\jProd{x}{y}}{z} = \inProd{x}{\jProd{y}{z}}$,
\end{enumerate}
for all $x,y,z \in \jAlg$. We will denote the 
identity element of $\jAlg$ by $\stdInt$ and 
we recall that $\jProd{\stdInt}{x} = x$, for all $x \in \jAlg$. The cone of squares associated to a Jordan algebra is given by 
$$
\stdCone = \{x^2 \mid x \in \jAlg \}.
$$
Under this setting, $\stdCone$ becomes a symmetric cone, i.e., a homogeneous\footnote{A cone is homogeneous if for every $x,y \in \reInt \stdCone$ there is a linear bijection $Q$ such that $Q(x) = y$ and $Q(\stdCone) = \stdCone$.} self-dual cone.  Reciprocally, every symmetric cone arises as the cone of squares of some Euclidean Jordan algebra. Key examples of symmetric cones include the $n\times n$ positive 
semidefinite matrices $\PSDcone{n}$, the nonnegative orthant $\Re^n_+$ and the second order cone. %

We say that $c \in \jAlg$ is an \emph{idempotent} if 
$\jProd{c}{c} = c$. Morover, $c$ is \emph{primitive} if it is nonzero and there is no way of writing 
$
c = a+b,
$
with  nonzero idempotents $a$ and $b$ satisfying $\jProd{a}{b} = 0$. We can now state the spectral theorem.
\begin{theorem} [Spectral Theorem, see Theorem III.1.2 in \cite{FK94}]\label{theo:spec}
	Let $(\jAlg, \jProd{}{} )$ be a Euclidean Jordan algebra and let $x \in \jAlg$. Then there are	 primitive idempotents $c_1, \dots, c_r$ satisfying
	$c_1 + \cdots + c_r  = \stdInt$, $\jProd{c_i}{c_j} = 0$ for $i \neq j$ and  
	 unique real numbers $\eigJ _1, \ldots, \eigJ _r$ satisfying
	\begin{equation}		
	x = \sum _{i=1}^r \eigJ _i c_i \label{eq:dec}.
	\end{equation}
\end{theorem}
The $\lambda _i$ appearing in Theorem \ref{theo:spec} are called 
the \emph{eigenvalues} of $x$. We will write $\lambda _{\min}(x)$ and $\lambda _{\max}(x)$ for 
the minimum and maximum eigenvalues of $x$, respectively.
For an element $x \in \jAlg$, we define the \emph{rank} of $x$ as the number of nonzero eigenvalues.
The trace of $x$ is defined as the sum of eigenvalues, i.e., 
$$
\tr x = \sum _{i=1}^r \lambda _i.
$$
The rank of $\stdCone$ is defined by 
$$
\matRank \stdCone = \max \{ \matRank x \mid x \in \stdCone \}.
$$
With that, we have $\matRank \stdCone = r = \tr (\stdInt)$. 

Throughout Section~\ref{sec:sym} and its subsections, we will assume that $\jAlg$ is a Euclidean Jordan algebra and that the inner product is given by 
\begin{equation}
\inProd{x}{y} = \tr(\jProd{x}{y}). \label{eq:def_inn}
\end{equation}
With that, the corresponding norm is
\begin{equation}
\norm{x} = \sqrt{\tr(x^2)} = \left({\sum _{i=1}^r \lambda _i^2}\right)^{1/2}\label{eq:norm}
\end{equation}
Under this inner product, the primitive idempotents $c_i$ appearing 
in Theorem~\ref{theo:spec} satisfy $\inProd{c_i}{c_j} = 0$ for $i \neq j$ and $\norm{c_i} = 1$.

The next result follows
from various propositions that appear in 
\cite{FK94}, such as Proposition~III.2.2 and Exercise 3 in Chapter III. See also Equation (10) in 
\cite{Sturm2000}. 
\begin{proposition}\label{prop:aux}
	Let $x \in \jAlg$.
	\begin{enumerate}[label=$(\roman*)$]
		\item $x \in \stdCone$ if and only if the eigenvalues of $x$ are nonnegative. \label{paux:1}
		\item $x \in \reInt \stdCone$ if and only if the eigenvalues of $x$ are positive. \label{paux:2}
		\item Suppose $x,y \in \stdCone$. Then, $\jProd{x}{y} = 0$ if and only if 
		$\inProd{x}{y} = 0$. \label{paux:4}
	\end{enumerate}
\end{proposition}

We will also need the following well-known fact on the function $\dist(\cdot, \stdCone)$.
Given $x \in \jAlg$, we consider the spectral decomposition given by Theorem~\ref{theo:spec}. Then, the element in $\stdCone$ closest to $x$ is given by 
$$
 y = \sum _{i=1}^r \max(\lambda _{i},0)c_i,
$$
where the $c_i$ are the primitive idempotents associated to 
$\lambda _{i}$ (a proof can be found in Proposition 2.2 of \cite{LFF17}). Therefore,

\begin{equation}
	\dist(x,\stdCone)^2 = \sum _{i=1}^r \max(-\lambda _{i}(x),0)^2. \label{prop:dist_sym_2}
\end{equation}
Given $x \in \jAlg$, the \emph{Lyapunov operator} of $x$ is the linear function $L_x : \jAlg \to \jAlg$ satisfying 
$L_x(y) = \jProd{x}{y}$, for all $y \in \jAlg$. The \emph{quadratic representation} of $x$ is the linear function $Q_x:\jAlg \to \jAlg$ such that $Q_x = 2L_x^2 - L_{x^2}$. 
We have
\begin{align}
Q_x(\stdInt) = x^2, \qquad \forall x \in \jAlg. \label{eq:sym_quad}
\end{align}
Let $c$ be an idempotent and $\alpha \in \Re$.
We define the following linear subspace of $\jAlg$.
$$
V(c,\alpha) = \{x \in \jAlg \mid \jProd{c}{x} = \alpha x \}.
$$

\begin{theorem}[Peirce Decomposition, see Proposition IV.1.1 and pg.~64 in \cite{FK94}]\label{theo:peirce1}
	Let $c \in \jAlg$ be an idempotent. Then $\jAlg$ is decomposed as the orthogonal direct sum
	$$
	\jAlg = V(c,1) \bigoplus V\left(c,\frac{1}{2}\right) \bigoplus V(c,0).
	$$	
	In addition, $V(c,1)$ and $V(c,0)$ are  Euclidean Jordan algebras under the same Jordan product $\jProd{}{}$. 
	The orthogonal projections onto $V(c,1)$ and $V(c,0)$ are given by $Q_c$ and $Q_{e-c}$, respectively.
	Furthermore, $\jProd{V(c,1/2)}{V(c,1/2)} \subseteq V(c,1)+V(c,0)$.
\end{theorem}
We conclude this review with our assumptions  for Section~\ref{sec:sym}.
\begin{assumption}[Overall assumptions for Section~\ref{sec:sym}]\label{assum:sym}
Throughout Section~\ref{sec:sym}, $\jAlg$ is a Euclidean Jordan algebra, $\stdCone$ is its cone of squares, the inner product is given by \eqref{eq:def_inn}, the norm is given by \eqref{eq:norm} and the distance function is the one induced by 	\eqref{eq:norm}.
\end{assumption}

\subsubsection{Facial structure of symmetric cones}\label{sec:sym_face}
One important property of symmetric cones is that all faces can be seen as smaller symmetric cones. To explain that, we first take an arbitrary idempotent $c$. Then, the algebras $V(c,1)$ and $V(c,0)$ appearing in Theorem~\ref{theo:peirce1} also give rise to symmetric cones. In fact, if we define
$$
\stdFace = \{x^2 \mid x \in V(c,1)\},
$$
we have that $\stdFace$ is a face of $\stdCone$ and $\spanVec \stdFace = V(c,1)$. As $\stdFace$ is the cone of squares of $V(c,1)$, it is also a symmetric cone on its own right. 
Therefore, it must be self-dual in some sense. However, if $\stdFace$ is a proper face of $\stdCone$ then it cannot possibly satisfy $\stdFace^* = \stdFace$. 
The correct way of understanding the self-duality of $\stdFace$ is by restricting ourselves to  $V(c,1)$. It holds that  $$\stdFace^* \cap V(c,1) = \stdFace.$$
Because of that, we say that $\stdFace$ is \emph{self-dual on its span}. 
Since all faces of $\stdCone$ are self-dual on their span, $\stdCone$ is a \emph{perfect cone}, following the definition by Barker \cite{Barker78}.

Under these conditions, we have $c \in \reInt \stdFace$ and $c$ is the identity element in $V(c,1)$. The conjugate face of $\stdFace$ is given as follows
$$
\stdFace ^\Delta = \stdCone \cap \{c\}^\perp = \{x^2 \mid x \in V(c,0)\}.
$$
That is, the faces generated by the algebras $V(c,0)$ and $V(c,1)$ are conjugate to each other. We remark that  $\stdInt - c$ is the identity element in $V(c,0)$ and $\spanVec \stdFace ^{\Delta} = V(c,0)$. 

Reciprocally, given a face $\stdFace$ of $\stdCone$, there exists an idempotent $c$ such that $\stdFace$ is the cone of squares of $V(c,1)$. We summarize these facts in the next proposition, which is a consequence of Theorem~2 in \cite{FB06}, due to Faybusovich.

\begin{proposition}\label{prop:sym_face}
	Let $\stdCone$ be a symmetric cone and $\stdFace$ be a face of $\stdCone$.
	\begin{enumerate}[label=({\it \roman*})]
		\item There is an idempotent $c\in \reInt \stdFace$ such that $V(c,1)$ is a Euclidean Jordan algebra, $\stdFace$ is the cone of squares of $V(c,1)$ and $\spanVec \stdFace = V(c,1)$.
		\item Let $c$ be as in the previous item. The conjugate face of $\stdFace$ is 
		$\stdFace^{\Delta} = \stdCone \cap \{c\}^\perp$ and is the cone of squares of 
		$V(c,0)$. Furthermore $\spanVec \stdFace^{\Delta} = V(c,0)$ and 
		$V(c,0) = V(\stdInt - c, 1)$.
		\item $\stdFace$ is self-dual on its span, i.e., $\stdFace = \stdFace^* \cap \spanVec \stdFace$.
	\end{enumerate}
\end{proposition}

Let $x  \in \jAlg$. If there exists $x^{-1}$ such that 
$\jProd{x}{x^{-1}} = \stdInt$ and $L_xL_{x^{-1}}=L_{x^{-1}}L_x $, 
we say that $x^{-1}$ is the inverse of $x$ in $\jAlg$, see Chapter~III of \cite{K99}.  
A sufficient condition for the existence of $x^{-1}$ is 
``$x \in \reInt \stdCone$''.
As in the case of symmetric matrices,  the eigenvalues of 
$x^{-1}$ are the reciprocals of the eigenvalues of $x$.

Now, let $c$ be an  idempotent and consider the algebra $V(c,1)$ together with its cone of squares $\stdFace$. 
If $x \in V(c,1)$, $x$ might have an inverse in $V(c,1)$ even if 
it does not have an inverse in $\jAlg$. In this case, 
$x^{-1}$ would satisfy $\jProd{x}{x^{-1}} = c$. Similarly, ``$x \in \reInt \stdFace$'' is a sufficient 
condition for the existence of an inverse in $V(c,1)$. With that, we have the following proposition.

\begin{proposition}\label{prop:schur}
	Let $\stdCone$ be a symmetric cone, $x \in \jAlg$ and $c$ be an idempotent. Following Theorem \ref{theo:peirce1}, write 
	$$
	x = x_1 + x_2 + x_3,
	$$
	with $x_1 \in V(c,1), x_2 \in V(c,1/2), x_3 \in V(c,0)$. Let $\stdFace$ be the cone of 
	squares of $V(c,1)$.
	\begin{enumerate}[label=$(\roman*)$]
		\item If $x \in \stdCone$, then 
		$x_1 \in \stdFace$ and $x_3 \in \stdFace^{\Delta}$.
		\item If $x \in \reInt \stdCone$, then 
		$x_1 \in \reInt \stdFace$ and $ x_3 \in \reInt\stdFace^{\Delta}$.
		\item (Schur complement
		) Suppose $x_3 \in \reInt \stdFace^{\Delta}$. Then 
		$x \in \reInt \stdCone$ if and only if 
		$$x_1 - Q_{x_2}(x_{3}^{-1}) \in \reInt \stdFace,$$
		where $x_3^{-1}$ denotes the inverse of $x_3$ in $V(c,0)$.
	\end{enumerate}
\end{proposition}
\begin{proof}
	\begin{enumerate}[label=$(\roman*)$]
		\item Let $y \in \stdFace$. Since 
		$x \in \stdCone$, we have $\inProd{x}{y} = \inProd{x_1}{y} \geq 0$. This shows that $x_1 \in \stdFace^* \cap V(c,1)$. 
		Since $\stdFace$ is self-dual over its span, we conclude that $x_1 \in \stdFace$.
		A similar argument holds for $x_3$.
		\item Let $y \in  \stdFace \setminus \{0\}$. 
		Since 
		$x \in \reInt \stdCone$, we have $\inProd{x}{y} = \inProd{x_1}{y} > 0$. This shows that $x_1 \in \reInt(\stdFace^* \cap V(c,1))$
		Since $\stdFace$ is self-dual over its span, we conclude that $x_1 \in \reInt \stdFace$.
		A similar argument holds for $x_3$.
		
		\item See Corollary 5 in the article by Gowda and Sznajder \cite{GS10}. 	
		
	\end{enumerate}
\end{proof}	

\subsubsection{Amenability and facial residual functions for symmetric cones}\label{sec:sym_frf}
We will first show that  symmetric cones are amenable.
\begin{proposition}[Symmetric cones are amenable and orthogonal projectionally exposed]\label{prop:sym_am}
	Let $\stdFace \face \stdCone$, where $\stdCone$ is a symmetric cone. There exists an orthogonal 
	projection $Q$ such $Q(\stdCone) = \stdFace$. In particular, 
	$\stdCone$ is amenable and 
	we have
	$$
	\dist(x,\stdFace) = \dist(x,\stdCone), \quad \forall x\in \spanVec \stdFace.
	$$
\end{proposition}
\begin{proof}
	Let $\stdCone$ be a symmetric cone and $\stdFace$ be a face of $\stdCone$. 
	First, we observe that  since $\stdFace \subseteq \stdCone$, we 
	have $\dist(x,\stdFace) \geq \dist(x,\stdCone)$, for every $x \in \jAlg$.
	
	Next, following Proposition~\ref{prop:sym_face}, 
	let $c$ be an idempotent such that $V(c,1)$ is the Euclidean Jordan algebra whose cone of squares is $\stdFace$ and such that $\spanVec \stdFace = V(c,1)$. 
	By Theorem~\ref{theo:peirce1}, the orthogonal projection onto $V(c,1)$ is given by $Q_c$. Together with item $(i)$ of Proposition~\ref{prop:schur}, we obtain that 
	$$Q_c(x) \in \stdFace, \quad \forall x \in \stdCone.
	$$
	This shows that $\stdCone$ is orthogonal projectionally exposed. Since 
	$\norm{Q_c} \leq 1$, we have that
	item $(i)$ of Proposition~\ref{prop:beaut} implies $\dist(x,\stdFace) = \dist(x,\stdCone)$ for all $x \in \spanVec \stdFace$. 
	
	
\end{proof}
Next, we will show that  symmetric cones admit {\FRFa}s of 
the form $\kappa \epsilon + \kappa \sqrt{\epsilon \norm{x}}$, where 
$\kappa$ is some positive constant. 
We first need a few auxiliary results.
%
%

\begin{lemma}\label{lemma:half}
	Let $\jAlg$ be a Euclidean Jordan algebra, let $c$ be an idempotent and $w \in V(c,1/2)$. Then, there are 
	$w_0 \in V(c,0)$, $w_1 \in V(c,1)$ such that
	\begin{align*}
	w^2 = w_0 + w_1\\
	\tr({w_0}) = \tr({w_1}) =\frac{\tr{(w^2)}}{2}.
	\end{align*}
\end{lemma}
\begin{proof}
	From Theorem~\ref{theo:peirce1}, we can 
	write $w^2 = w_0 + w_1$, with $w_0 \in V(c,0)$, $w_1 \in V(c,1)$.
	On one hand, taking the inner product with $c$, we obtain
	\begin{align*}
	\inProd{w^2}{c} = \inProd{w}{\jProd{w}{c}} = \frac{\inProd{w}{w}}{2} = \frac{\tr(w^2)}{2},
	\end{align*}
	where the first equality follows from axiom $(3)$ in Section~\ref{sec:sym} and the second equality follows from the assumption that $w \in V(c,1/2)$.
	On the other hand, we have $$\inProd{w_0+w_1}{c} = \inProd{w_1}{c},$$ since $w_0 \in V(c,0)$.
	To conclude, we recall that $\stdInt-c$ belongs 
	to $V(c,0)$, so that $$\tr(w_1) = \inProd{w_1}{c + (\stdInt-c)}=\inProd{w_1}{c}. $$
	
\end{proof}
At last, we recall the following variational characterization of $\lambda_{\min}$, which can be found, for instance, in Equation~(9) in \cite{Sturm2000}.
$$
\lambda_{\min}(x) = \min \{ \inProd{x}{y} \mid y \in \stdCone, \inProd{y}{\stdInt} = 1 \}.	
$$
Then, since $\tr(y) = \inProd{y}{\stdInt}$, we obtain 
\begin{equation}\label{eq:in_bound}
\inProd{x}{y} \geq \lambda _{\min}(x) \tr(y), \quad\forall x \in \jAlg, \forall y \in \stdCone.
\end{equation}

\begin{theorem}[Facial residual functions for symmetric cones]\label{theo:sym_am}%
	Let $\stdCone$ be a symmetric cone and let $\stdFace \face \stdCone$ be an arbitrary face.
	Let $z \in \stdFace^*$ and 
	$\hat \stdFace = \stdFace \cap \{z\}^{\perp}$.
	Then, there is a positive constant $\kappa$ (depending  on $\stdCone, \stdFace, z$) such that 
	whenever $x$ satisfies  the inequalities
	$$
	\dist(x,\stdCone) \leq \epsilon, \quad \inProd{x}{z} \leq \epsilon, \quad \dist(x, \spanVec \stdFace ) \leq \epsilon
	$$
	we have
	$$
	\dist(x,\hat \stdFace)  \leq \kappa \epsilon +  \kappa \sqrt{\epsilon \norm{x}}.
	$$
	That is, we can take $\psi _{ \stdFace,z}(\epsilon,\norm{x}) = \kappa \epsilon +  \kappa \sqrt{\epsilon \norm{x}}$ as a {\FRF} for $\stdFace$ and $z$.
\end{theorem}
\begin{proof}
	Let $\stdFace$ be a face of $\stdCone$, $z \in \stdFace^*$ and let $\hat \stdFace = \stdFace \cap \{z\}^\perp$.
	By item $(i)$ of Proposition \ref{prop:sym_face},
	there is an idempotent $c \in \reInt \stdFace$ such that $V(c,1)$
	is a Jordan algebra satisfying
	$$
	\stdFace = \{u^2 \mid u \in V(c,1)\}.
	$$
	Furthermore, we have  $V(c,1) = \spanVec \stdFace$.
	Now, suppose that we have $x \in \jAlg$ such that 
	$$
	\dist(x,\stdCone) \leq \epsilon, \quad \inProd{x}{z} \leq \epsilon, \quad \dist(x, \spanVec \stdFace ) \leq \epsilon.
	$$	
	By Theorem \ref{theo:peirce1}, we can decompose $x$ and $z$ as 
	\begin{align*}
	x &= x_1 + x_2 + x_3\\
	z &= z_1 + z_2 + z_3,
	\end{align*}
	where $x_1,z_1 \in V(c,1), x_2,z_2 \in V(c,1/2), x_3,z_3 \in V(c,0)$.
	We recall that $V(c,1)$, $V(c,1/2)$ and $V(c,0)$ are orthogonal subspaces. In particular, this implies that  $x_1$ is the orthogonal projection of $x$ onto 
	$V(c,1)= \spanVec \stdFace$. 
	Therefore, $\dist(x, \spanVec \stdFace ) \leq \epsilon$
	implies  $\norm{x - x_1} \leq \epsilon$. As $x_2$ and $x_3$ are orthogonal, we obtain
	\begin{align}
	\norm{x_2} &\leq \epsilon \label{eq:bd2}\\
	\norm{x_3} &\leq \epsilon \label{eq:bd3}. 
	\end{align}
	Now we turn our attentions to $\hat \stdFace$.
	First, since $z \in \stdFace ^*$, we have 
	\begin{equation}\label{eq:sym_am_z1}
	z_1 \in \stdFace,\qquad \hat \stdFace = \stdFace \cap \{z_1\}^{\perp}. \footnote{To see that, first 
		recall that $\stdFace$ is self-dual on its span, i.e., $\stdFace = \{v \in V(c,1)\mid \inProd{u}{v}\geq 0, \forall u \in \stdFace \}$.
		Now, let $v \in \stdFace$ be arbitrary. Since $z \in \stdFace^*$, we have $\inProd{v}{z} = \inProd{v}{z_1} \geq 0 $, due to the orthogonality among $V(c,0),V(c,1/2)$ and $V(c,1)$.
		This shows that $z_1 \in \stdFace$. Similarly, we can show that $\stdFace \cap\{z\}^\perp = \stdFace \cap\{z_1\}^\perp.  $}
	\end{equation}	
	As $V(c,1)$ is a \emph{bona fide} Jordan algebra and $\hat \stdFace$ is a face of $\stdFace$, again by Proposition~\ref{prop:sym_face} there is some idempotent $\hat c$ such 
	that $\hat V (\hat c, 1) $ is the Jordan algebra contained in 
	$V(c,1)$  that generates $\hat \stdFace$, i.e., 
	$$
	\hat \stdFace = \{u^2 \mid u \in \hat V (\hat c, 1)\},
	$$
	where $$\hat V(\hat c, \alpha) = \{u \in V(c,1) \mid \jProd{\hat c}{u} = \alpha u \} = V(\hat c, \alpha) \cap V(c,1).$$
	We remark that 
	$\hat V(\hat c, \alpha)$ might be smaller than 
	$V(\hat c, \alpha)$ and we use the symbol $\hat V$ to emphasize that $\hat V(\hat c, \alpha)$ is a subalgebra of $V(c,1)$.
	
	Given the idempotent $\hat c$, we  apply Theorem~\ref{theo:peirce1}
	substituting $\jAlg$ by $V(c,1)$ and $c$ by $\hat c$. It follows that
	$$V(c,1) = \hat V(\hat c, 1) \oplus \hat V(\hat c,1/2) \oplus \hat V(\hat c, 0).$$
	Then, we further decompose $x_1$ as
	$$
	x_1 = x_{11} + x_{12} + x_{13},
	$$
	with $x_{11} \in \hat V(\hat c, 1), x_{12}\in \hat V(\hat c, 1/2), x_{13} \in \hat V(\hat c, 0)$. 
	Our goal is to bound to $x_{12}$ and $x_{13}$.

	We first bound $x_{13}$ by invoking Lemma~\ref{lemma:reint_dual} appropriately. 
	To do so, first recall \eqref{eq:sym_am_z1}, so that $z_1 \in \stdFace$ and $\hat \stdFace = \stdFace \cap \{z_1\}^\perp$. 
	We restrict ourselves to $V(c,1)$ and let $\hat \stdFace ^{\Delta}$ denote the conjugate face of $\hat \stdFace$ with respect 
	to 	$\stdFace$. That is, $$\hat \stdFace ^{\Delta} = \stdFace \cap \hat \stdFace ^\perp = V(c,1)\cap\stdFace^*\cap \hat \stdFace ^\perp.  $$	
	Recalling that $\inProd{x}{z} \leq \epsilon$, we have
	\begin{align}
	\inProd{z_1}{x_1} & \leq  \epsilon - \inProd{z_2}{x_2} - \inProd{z_3}{x_3}  && \notag \\
	& \leq \epsilon + \epsilon \norm{z_2} + \epsilon \norm{z_3}&& \text{(From \eqref{eq:bd2} and \eqref{eq:bd3})} \notag\\
	& \leq \epsilon(1+\norm{z_2} + \norm{z_3}). \label{eq:bd_z1x1} &&
	\end{align}
	Recall that symmetric cones are nice because they are amenable (Proposition~\ref{prop:am_nice}), see also Proposition~4 and Section~4.1 in the work by Chua and Tun\c{c}el~\cite{CT08} or  Theorem~4.1 in the work by P\'olik and Terlaky~\cite{PT07}.
	Since $\stdFace$ is a symmetric cone (see Section~\ref{sec:sym_face}), $\stdFace$ must be nice as well. 
	Therefore, we can
	apply  Proposition~\ref{prop:cj_ri} to $\stdFace$ and $z_1$. 
	This shows that  $z_1 \in \reInt \hat \stdFace ^{\Delta}$ and, in particular, $z_1 \in \hat V(\hat c, 0)$.\footnote{Rigorously, the argument so far only shows that $z_1 \in \reInt(\stdFace^*\cap \hat \stdFace^\perp)$. However, since $z_1 \in V(c,1)$, we can put ``$\reInt$'' outside and conclude that $V(c,1)\cap \reInt(\stdFace^*\cap \hat \stdFace^\perp) = \reInt(\stdFace^*\cap \hat \stdFace^\perp\cap V(c,1))$. Therefore, as remarked, $z_1 \in \reInt \hat \stdFace^{\Delta}$. Furthermore, since $z_1 \in \stdCone$ and $\hat c \in \hat \stdFace$, we have $\inProd{\hat c}{z} = 0$. By item $(iii)$ of Proposition~\ref{prop:aux}, we have $\jProd{\hat c}{z} = 0$ and $z_1 \in \hat V(\hat c, 0)$ as claimed. }
	As $\hat \stdFace ^{\Delta}$ is a symmetric cone whose Jordan algebra is $\hat V(\hat c,0) $, we 
	have $\hat \stdFace ^{\Delta} = \hat \stdFace ^{\Delta*}\cap \hat V(\hat c,0)$, by items $(ii)$ and $(iii)$ of Proposition \ref{prop:sym_face}. 
	This shows that  $z_1 \in \reInt \hat \stdFace ^{\Delta*}$. 
	
	Now, since $\dist(x,\stdCone) \leq \epsilon$, we 
	have\footnote{Let $u \in \stdCone$ be such that $\dist(x,\stdCone) = \norm{x-u}$. Decompose $u$ following the same decomposition of $x$. We have $u = u_{11} + u_{12} + u_{13} + u_2 + u_3$. By item $(i)$ of Proposition~\ref{prop:schur}, we have that $u_{13} \in \hat \stdFace ^{\Delta}$. Therefore $\dist(x_{13}, \hat \stdFace ^{\Delta}) \leq \norm{x_{13} - u_{13}} \leq \norm{x-u} \leq \epsilon$.  Similarly, we have $\dist(x_{1}, \stdFace ) \leq \norm{x_{1} - u_{1}} \leq \epsilon$. } 
	\begin{align}
	\dist(x_{1}, \stdFace) \leq \epsilon, \quad \dist(x_{13}, \hat \stdFace ^{\Delta}) \leq \epsilon. \label{eq:aux}
	\end{align}
	Therefore, there is $u \in \hat V(\hat c, 0)$ such that $x_{13} + u \in \hat \stdFace^{\Delta}$ and $\norm{u} \leq \epsilon$. Since 
	$z_1 \in \hat V(\hat c,0)$, we have the following inequalities
	\begin{align*}
	\inProd{z_{1}}{x_{13} + u} & = \inProd{z_{1}}{x_{11} + x_{12} + x_{13} + u} &&\text{($z_1$ is orthogonal to $x_{11},x_{12}$)}\\
	& = \inProd{z_{1}}{x_1 + u} && \\ 
	& \leq \epsilon(1+\norm{z_2} + \norm{z_3} +\norm{z_1})&& \text{(From \eqref{eq:bd_z1x1})}.
	\end{align*}
	We apply Lemma~\ref{lemma:reint_dual} to 
	$\hat \stdFace ^{\Delta}$ and $z_1$, which tells us that there is $\kappa _1 > 0$ such that $\norm{w} \leq \kappa _1 \inProd{w}{z_1} $ whenever $w \in \hat \stdFace ^{\Delta}$.
	It follows that $$\norm{x_{13}+u} \leq  \epsilon\kappa _1(1+\norm{z_2} + \norm{z_3} +\norm{z_1}). $$ 
	As $\norm{u} \leq \epsilon$, we conclude 
	that 
	\begin{equation}\label{eq:bd13}
	\norm{x_{13}} \leq \hat \kappa _1\epsilon,
	\end{equation}
	where $\hat \kappa _1 = (\kappa _1((1+\norm{z_2} + \norm{z_3} +\norm{z_1}) +1)$.
	
	The next task is to bound $x_{12}$.	First, we 
	apply Lemma \ref{lemma:half} to $x_{12}$, with 
	$V(c,1)$ in place of $\jAlg$, thus obtaining 
	$w_0 \in \hat V(\hat c,0)$ and $w_1 \in \hat V(\hat c,1)$ such that
	\begin{align}
	x_{12}^2 & = w_0 + w_1  \label{eq:x12_1} \\
	\tr(w_0) & = \frac{\tr(x_{12}^2)}{2}.\label{eq:x12_2}
	\end{align}
	From \eqref{eq:aux}, and since  $c$ is the identity element in $V(c,1)$,  we have $$x_1 + \epsilon c \in \stdFace.\footnote{
		If $x_1 \in \stdFace$, then we have $x_1 + \epsilon c \in \stdFace$. If not, then $\lambda _{\min}(x_1) < 0$. Here, we are considering the minimum eigenvalue of $x_1$ with respect the algebra $V(c,1)$.
		In this case, from Proposition~\ref{prop:dist_sym_2}, we have that $\epsilon \geq \dist(x_{1}, \stdFace)\geq -\lambda _{\min}(x_1)$. Then, since $c$ is the identity in $V(c,1)$, adding $\epsilon c$ to $x_1$ has the effect of adding $\epsilon$ to $\lambda _{\min}(x_1)$.  }$$ 
	In addition, since $c \in \reInt \stdFace$, the following holds for every $\alpha > 0$, 
	$$x_1 + (\alpha + \epsilon) c \in \reInt \stdFace.$$
	We write $c = \hat c + (c - \hat c)$ and recall that $\hat c \in V(\hat c,1)$ and $(c-\hat c) \in \hat V(\hat c,0)$. Then, we obtain from item $(ii)$ of Proposition \ref{prop:schur} that 
	\begin{align}
	x_{11}+ (\epsilon +\alpha)\hat c & \in \reInt \hat \stdFace \label{eq:x_11}\\
	x_{13} + (\epsilon +\alpha)(c-\hat c) & \in \reInt \hat \stdFace ^{\Delta}. \notag
	\end{align}
	Now, we apply item $(iii)$ of Proposition \ref{prop:schur}, which tells us that 
	the following Schur complement must be a relative interior point of $\hat \stdFace$.
	\begin{equation}
	(x_{11}+ (\epsilon +\alpha)\hat c) - Q_{x_{12}}((x_{13} + (\epsilon +\alpha)(c-\hat c) )^{-1}) \in \reInt \hat \stdFace, \label{eq:schur_goal} 
	\end{equation}
	where $(x_{13} + (\epsilon +\alpha)(c-\hat c) )^{-1}$ is the inverse of $x_{13} + (\epsilon +\alpha)(c-\hat c)$ in $\hat V(\hat c, 0)$.
	
	The next subgoal is to bound from below the minimum eigenvalue\footnote{The subtlety here is that $x_{13} + (\epsilon +\alpha)(c-\hat c) $ and its inverse, seen as elements of $\hat V(\hat c,0)$, have no zero eigenvalues, since they belong to $\reInt \hat \stdFace^{\Delta}$. If we see them as elements of $\jAlg$, zero eigenvalues might appear, but the corresponding idempotents certainly do not belong to $\hat V(\hat c,0)$. } of $(x_{13} + (\epsilon +\alpha)(c-\hat c) )^{-1}$ in the algebra $\hat V(\hat c, 0) $.	 Since $x_{13} + (\epsilon +\alpha)(c-\hat c) \in \hat \stdFace ^{\Delta}$ and $c-\hat c$ is the unit element in $\hat V(\hat c, 0)$, we have 
	$$
	\lambda _{\max}(x_{13} + (\epsilon +\alpha)(c-\hat c)) = \lambda _{\max}(x_{13}) + \epsilon + \alpha.
	$$
	In addition, from \eqref{eq:bd13} and \eqref{eq:norm}, we have that $\lambda _{\max}(x_{13}) \leq \hat \kappa _1 \epsilon$. Thus, we obtain 
	\begin{equation}\label{eq:bd13_inv}
	\lambda _{\min}(x_{13} + (\epsilon +\alpha)(c-\hat c) )^{-1} = \frac{1}{\lambda _{\max}(x_{13} + (\epsilon +\alpha)(c-\hat c) )} \geq \frac{1}{(\hat \kappa _1 +1)\epsilon + \alpha}.
	\end{equation}
	We now return to \eqref{eq:schur_goal}. As the Schur complement is a relative interior point of $\hat \stdFace^{\Delta}$, its inner product with $c$ must be nonnegative. Recalling that $Q_{x_{12}}$ is self-adjoint, it follows that 
	\begin{align*}
	\inProd{x_{11}+ (\epsilon +\alpha) \hat c}{c}  & \geq \inProd {Q_{x_{12}}((x_{13} + (\epsilon +\alpha)(c - \hat c) )^{-1} )}{c} && \\
	& = \inProd {(x_{13} + (\epsilon +\alpha)(c - \hat c) )^{-1} }{x_{12}^2} && \text{(From \eqref{eq:sym_quad})  }\\
	& = \inProd {(x_{13} + (\epsilon +\alpha)(c - \hat c) )^{-1} }{w_0} && \text{(From \eqref{eq:x12_1})  }\\
	& \geq \lambda _{\min}((x_{13} + (\epsilon +\alpha)(c - \hat c) )^{-1})\tr{(w_{0})} && \text{(From \eqref{eq:in_bound})\footnotemark}\\
	& = \lambda _{\min}((x_{13} + (\epsilon +\alpha)(c - \hat c) )^{-1})\frac{\tr{(x_{12}^2)}}{2} && \text{(From \eqref{eq:x12_2})}\\
	& \geq \frac{1}{(\hat \kappa _1 +1)\epsilon + \alpha} \frac{\norm{x_{12}}^2}{2}. && \text{(From \eqref{eq:bd13_inv} and \eqref{eq:norm})}
	\end{align*}\footnotetext{\eqref{eq:in_bound} is invoked with $\hat V(\hat c, 0)$ in place of $\jAlg$, so 
		that $\lambda _{\min}((x_{13} + (\epsilon +\alpha)(c - \hat c) )^{-1})$ refers to the minimum eigenvalue in the algebra $\hat V(\hat c, 0)$ and that is also why we can use \eqref{eq:bd13_inv} at the end.}%
Using the Cauchy-Schwarz inequality, we 
	get
	$$
	\norm{x_{12}}^2 \leq  2((\hat \kappa _1 +1)\epsilon + \alpha)\norm{c}\norm{x_{11} + (\epsilon +\alpha) \hat c}.
	$$
	Since $\alpha$ is an arbitrary positive number, we get
	$$
	\norm{x_{12}}^2 \leq 2((\hat \kappa _1 +1)\epsilon)\norm{c}\norm{x_{11} + \epsilon \hat c}.
	$$
	Therefore,
	\begin{equation*}
	\norm{x_{12}} \leq  \hat \kappa_2\sqrt{\norm{\epsilon x_{11} + \epsilon^2 \hat c}},
	\end{equation*}
	where $\hat \kappa _2 = \sqrt{2(\hat \kappa _1 +1)\norm{c}}$. Finally, using the triangle inequality, we 
	get
	\begin{equation}\label{eq:bd12}
	\norm{x_{12}} \leq  \epsilon \hat \kappa _2 \sqrt{\norm{\hat c}}   +  \hat\kappa _2\sqrt{\epsilon \norm{x_{11}}}.
	\end{equation}
	We are now ready to bound 
	$\dist(x, \hat \stdFace)$. From \eqref{eq:x_11}, we have that 
	$x_{11}- \epsilon \hat c \in \hat \stdFace$. It follows that
	\begin{align*}
	\dist(x,\hat \stdFace) & \leq \norm{x - x_{11}- \epsilon \hat c} \\
	& \leq \norm{x_{12} + x_{13} + x_{2} + x_{3}} + \epsilon \norm{\hat c}\\
	& \leq \epsilon \hat \kappa _2 \sqrt{\norm{\hat c}}   +  \hat\kappa _2\sqrt{\epsilon \norm{x_{11}}} + \hat \kappa _1\epsilon + 2\epsilon + \epsilon \norm{\hat c}&& \text{(From \eqref{eq:bd2}, \eqref{eq:bd3}, \eqref{eq:bd13}, \eqref{eq:bd12})}\\
	& \leq \kappa \epsilon +   \kappa \sqrt{\epsilon \norm{x}},
	\end{align*}
	where $\kappa = \max(\hat \kappa _2 \sqrt{\norm{\hat c}} + \hat \kappa _1 + 2 + \norm{\hat c},\hat \kappa _2)$.	
\end{proof}

\subsubsection{H\"olderian error bounds for symmetric cones}

Following Theorem \ref{theo:sym_am}, our first step is to bound by above the composition of facial residual functions of $\stdCone$.
\begin{lemma}\label{lem:sym_res}
Suppose that $\psi _i(\epsilon,\norm{x}) = \kappa _i \epsilon + \kappa _i\sqrt{\epsilon \norm{x}}$ for $i = 1,\ldots, \ell-1$, where the $\kappa_i$ are positive constants and $\ell \geq 2$. Then, there is a positive constant $\kappa$ such that
$$
\psi _{{\ell-1}}\comp \ldots \comp \psi_{{1}}(\epsilon,\norm{x}) \leq \kappa\sum _{j = 0}^{\ell-1}  \epsilon^{({2^{-j})} }\norm{x}^{1-{2^{-j}} },
$$	
for every $\epsilon \geq 0$ and every $x$.
\end{lemma}
\begin{proof}
We proceed by induction on $\ell$. 
For $\ell = 2$, it is enough to take $\kappa = \kappa _1$.
Now, suppose that the proposition is true for some $\ell > 2$. We will show that it is also true for $\ell + 1$. 
Let $\varphi = \psi _{{\ell}}\comp \ldots \comp \psi_{{1}}$. Since $\psi _{\ell}$ is monotone nondecreasing in each argument, we have by the induction hypothesis that there exists some $\tilde \kappa$ such that
\begin{align*}
 \varphi (\epsilon,\norm{x})& = \psi _{\ell}(\epsilon + \psi _{{\ell-1}}\comp \ldots \comp \psi_{{1}}(\epsilon,\norm{x}),\norm{x})\\
& \leq \psi _{\ell}\left(\epsilon + \tilde \kappa\sum _{j = 0}^{\ell-1} \epsilon^{({2^{-j}}) }\norm{x}^{1-{2^{-j}} },\norm{x}\right) \\
& \leq \kappa _{\ell}\epsilon +  \sum _{j = 0}^{\ell-1} \kappa _{\ell}\tilde \kappa \epsilon^{({2^{-j}}) }\norm{x}^{1-{2^{-j}}} + \kappa _{\ell}\sqrt{\epsilon\norm{x}}+
 \sum _{j = 0}^{\ell-1} \kappa _{\ell}\sqrt{\tilde \kappa} \epsilon^{({2^{-j-1}}) }\norm{x}^{1-({2^{-j-1}})},
 \end{align*}
where we used the fact that the square root satisfies $\sqrt{u + v} \leq \sqrt{u} + \sqrt{v}$, when $u$ and $v$ are nonnegative. Looking at the terms that appear in both summations, we see that it is possible to group the coefficients and so we obtain
\begin{align*}
\varphi (\epsilon,\norm{x}) \leq \kappa\sum _{j = 0}^{\ell}  \epsilon^{({2^{-j}}) }\norm{x}^{1-({2^{-j}}) },
\end{align*}
for $\kappa = \kappa _{\ell} + \kappa_{\ell}\tilde \kappa +\kappa_{\ell}\sqrt{\tilde \kappa}$.
\end{proof}

With that we have the following theorem.
\begin{theorem}[Error bounds for symmetric cones - 1st form]\label{theo:sym_err}
Let $\stdCone$ be a symmetric cone, $\stdSpace$ a subspace and $a \in \jAlg$ such that $(\stdCone,\stdSpace,a)$ is feasible.	
Then, there is a positive constant $\kappa$ (depending on $\stdCone,\stdSpace,a$) such that whenever  $x$ and $\epsilon$ satisfy the inequalities
$$
\quad \dist(x,\stdCone) \leq \epsilon, \quad \dist(x,\stdSpace + a) \leq \epsilon,
$$
we have 
	$$
	\dist\left(x, (\stdSpace + a) \cap \stdCone\right) \leq 
	(\kappa \norm{x} + \kappa)\left(\sum _{j = 0}^{\dpp(\stdSpace,a)}  \epsilon^{({2^{-j})} }\norm{x}^{1-{2^{-j}} }\right).
	$$
\end{theorem}
\begin{proof}
$\stdCone$ is an amenable cone, because of Proposition \ref{prop:sym_am}.
Therefore, we may apply Theorem \ref{theo:err} and Proposition \ref{prop:effic}, which tell us that there exists a positive constant $\tilde \kappa$ such that
\begin{equation}\label{eq:err_sym1}
	\dist\left(x, (\stdSpace + a) \cap \stdCone\right) \leq (\tilde \kappa \norm{x}+ \tilde \kappa) (\epsilon+\varphi(\epsilon,\norm{x}))
\end{equation}
where 
\[\varphi(\epsilon,\norm{x}) =
\begin{cases}
\psi _{\dpp(\stdSpace,a)} \comp \cdots \comp\psi _1(\epsilon,\norm{x})       & \quad \text{if } \dpp(\stdSpace,a) > 0\\
\epsilon  & \quad \text{if } \dpp(\stdSpace,a) = 0,
\end{cases}
\]
and the $\psi _i$ are 
{\FRFs} as in Theorem \ref{theo:sym_am}.	
Then, we apply Lemma \ref{lem:sym_res}, to obtain a constant $\kappa'$ such that 
$$
\varphi(\epsilon,\norm{x})\leq \kappa'\sum _{j = 0}^{\dpp(\stdSpace,a)}  \epsilon^{({2^{-j}}) }\norm{x}^{1-{2^{-j}} }.
$$
We then let $\hat \kappa = \kappa'+1$ so that 
\begin{equation}\label{eq:err_sym2}
\epsilon + \varphi(\epsilon,\norm{x}) \leq \hat \kappa \sum _{j = 0}^{\dpp(\stdSpace,a)} \epsilon^{({2^{-j}}) }\norm{x}^{1-{2^{-j}} }.
\end{equation}
Using \eqref{eq:err_sym2} in \eqref{eq:err_sym1}
and letting $\kappa = \tilde \kappa\hat \kappa$ gives the desired error bound. 
\end{proof}

We observe that if  $\stdCone$ is a symmetric cone and 
$d$ is taken to be the identity element $\stdInt$, then the generalized eigenvalue function $\eig{\stdInt}{\stdCone}(\cdot)$ discussed in Section~\ref{sec:eig} coincides with the minimum eigenvalue function $\lambda _{\min}(\cdot)$.
Following Remark~\ref{rem:dist}, we may also substitute 
the condition ``$\dist(x,\stdCone) \leq \epsilon$'' in Theorem \ref{theo:sym_err} by 
``$\lambda_{\min}(x) \geq - \epsilon$''.
Furthermore, if $x$ lies in some compact set and $\epsilon \leq 1$ then we can give a better looking error bound, where the sum is replaced by the term with smallest exponent. This leads to the second form of our error bounds results, which is closer to the way Sturm stated his error bound result.

\begin{proposition}[Error bounds for symmetric cones - 2nd form]\label{prop:sym_err2}
Let $\stdCone$ be a symmetric cone, $\stdSpace$ a subspace and $a \in \jAlg$ such that $(\stdCone,\stdSpace,a)$ is feasible.
Let $\rho$ be a positive real number.
Then, there exists a  positive constant $\kappa $ (depending on $\stdCone,\stdSpace,a,\rho$) such that for every $x$ and $\epsilon \leq 1$ satisfying	
$$
\quad \dist(x,\stdCone) \leq \epsilon, \quad \dist(x,\stdSpace + a) \leq \epsilon ,  \quad \norm{x} \leq \rho,
$$
we have 
$$
\dist\left(x, (\stdSpace + a) \cap \stdCone\right) \leq \kappa \epsilon^{({2^{-\dpp(\stdSpace,a)}})}.
$$
Furthermore, the proposition is still valid if we replace
the inequality ``$\dist(x,\stdCone) \leq \epsilon$'' by 
``$\lambda_{\min}(x) \geq - \epsilon$''.

\end{proposition}
\begin{proof}
We apply Theorem \ref{theo:sym_err} to $(\stdCone,\stdSpace,a)$. 
Let $ \hat \kappa$ be the obtained constant. 
Since $\epsilon \leq 1$, we have 
$$
\epsilon ^{({2^{-\dpp(\stdSpace,a)}})} \geq \epsilon ^{({2^{-j}})},
$$
for all $j = 0, \ldots, \ell - 1$. 
Recalling that $\norm{x} \leq \rho$, we have	\begin{align*}
	\dist\left(x, (\stdSpace + a) \cap \stdCone\right) & \leq 
	(\hat \kappa \norm{x} + \hat \kappa)\left(\sum _{j = 0}^{\dpp(\stdSpace,a)} \epsilon^{({2^{-j}}) }\norm{x}^{1-{2^{-j}} }\right)\\
	&\leq (\hat \kappa \rho + \hat \kappa)\left(\sum _{j = 0}^{\dpp(\stdSpace,a)}  \epsilon^{({2^{-\dpp(\stdSpace,a)})} }\rho^{1-{2^{-j}} }\right)\\
	& \leq \kappa \epsilon^{({2^{-\dpp(\stdSpace,a)}}) }, 
\end{align*}
where $\kappa$ is the square of the maximum among all the constants so far, that is,
$$
 \sqrt\kappa = \max\{\hat \kappa \rho + \hat \kappa,\max\{ \rho^{1-\frac{1}{2^j} }\mid j=0,\ldots,\dpp(\stdSpace,a) \}\}.
$$ 
The fact that error bound is still valid if we replace
the inequality ``$\dist(x,\stdCone) \leq \epsilon$'' by 
``$\lambda_{\min}(x) \geq - \epsilon$'' follows from Remark~\ref{rem:dist}. 
\end{proof}

\begin{remark}[Bounds on the distance to the PPS condition]\label{rem:sym}
  We 
	can use Proposition~\ref{prop:effic} to bound 
	the quantity $\dpp(\stdSpace,a)$ in both Theorem~\ref{theo:sym_err} and Proposition~\ref{prop:sym_err2}. For that, we need the following facts on a symmetric cone $\stdCone$.
	\begin{enumerate}[label=$(\roman*)$]
		\item The length $\ell_{\stdCone}$ of the longest chain of faces of $\stdCone$  satisfies $\ell_{\stdCone} = \matRank	\stdCone +1$.
		\item The distance to polyhedrality of $\stdCone$ satisfies 
		$\distP(\stdCone) \leq \matRank \stdCone-1$.
	\end{enumerate}
	Item $(i)$ is a result due to Ito and Louren\c{c}o, see Theorem 14 in \cite{IL17}.
	Then, Theorem~11 in \cite{LMT15} tell us that $1+\distP(\stdCone) \leq \ell _{\stdCone} -1$. It follows that $\distP(\stdCone) \leq \matRank \stdCone-1$, which is item $(ii)$. 	Therefore, if $\stdCone = \stdCone^1 \times \cdots \times \stdCone^s$ is the direct product of $s$ symmetric cones, we have the the following bound.
	$$
	\dpp(\stdSpace,a) \leq \min \left\{\dim (\stdSpace^\perp \cap \{a\}^\perp), \sum _{i=1}^{s} (\matRank \stdCone^i -1),\ds(\stdSpace,a) \right\}.
	$$
\end{remark}


%

\subsection{Intersection of cones}\label{sec:inter}
Suppose $\stdCone^1 \subseteq \jAlg$ and $\stdCone ^2\subseteq \jAlg$ are amenable cones.
It is not clear whether  $\stdCone^1 \cap \stdCone^2$ is also  amenable. Even if it turns out that $\stdCone^1 \cap \stdCone^2$ is indeed amenable, it is also not clear how to construct {\FRFa}s  
for  $\stdCone^1 \cap \stdCone^2$ from the {\FRFa}s of 
$\stdCone^1$ and $\stdCone^2$.
Therefore, at first glance, the results in Theorem \ref{theo:err} are not directly applicable.
Nevertheless, we will show in this subsection that it is still possible to give error bounds while sidestepping these issues.

Suppose $(\stdCone^1\cap \stdCone^2,\stdSpace,a)$ is feasible. 
Let $\hat \stdSpace, \hat a$ be such that 
$$
\hat \stdSpace + \hat a = \{ (x,x) \mid x \in \stdSpace+a \}.
$$
Due to Propositions \ref{prop:beaut_prev} and \ref{prop:ebtp}, $\stdCone^1 \times \stdCone^2$ is an amenable cone.
Furthermore, we can use as {\FRFa}s the sum of 
{\FRFs} for $\stdCone^1$ and $\stdCone^2$.
We will show in this subsection that it is possible to obtain error bounds for $(\stdCone^1\cap \stdCone^2,\stdSpace,a)$ through 
$(\stdCone^1\times \stdCone^2,\hat \stdSpace, \hat a) $.
Recall that, by convention (see Section \ref{sec:conv}), the inner product in $\jAlg \times \jAlg$ is such that if $(x,y),(\hat x, \hat y) \in \jAlg\times \jAlg$, we have $\inProd{(x,y)}{(\hat x, \hat y)} = \inProd{x}{\hat x} + \inProd{y}{\hat y}$. Then, for $x \in \jAlg$, it can be verified that
\begin{align}
\dist(x, \stdCone^1\cap \stdCone^2 ) \leq \epsilon & \Rightarrow \dist((x,x), \stdCone^1\times \stdCone^2 ) \leq \sqrt{2} \epsilon \label{eq:inter:1}\\
\dist(x, \stdSpace+a ) \leq \epsilon & \Rightarrow \dist((x,x), \hat \stdSpace + \hat a ) \leq \sqrt{2} \epsilon \label{eq:inter:2}\\
\dist(x, \stdCone^1\cap \stdCone^2 \cap (\stdSpace+a) ) & \leq \frac{1}{\sqrt{2}}\dist((x,x), (\stdCone^1\times \stdCone^2)\cap (\hat \stdSpace+\hat a) ). \label{eq:inter:3}
\end{align}
Then, the next proposition follows immediately from Theorem \ref{theo:err}.
\begin{proposition}[Error bound for the intersection of amenable cones]\label{prop:inter}
Suppose $\stdCone^1 \subseteq \jAlg$ and $\stdCone ^2\subseteq \jAlg$ are amenable cones. Suppose also that $(\stdCone^1\cap \stdCone^2,\stdSpace,a)$ is feasible.	Let $\hat \stdSpace, \hat a$ be such that 
$$
\hat \stdSpace + \hat a = \{ (x,x) \mid x \in \stdSpace+a \}.
$$
The following hold.
\begin{enumerate}[label=$(\roman*)$]
\item Let 
$$
\stdFace _{\ell}  \subsetneq \ldots \subsetneq \stdFace_1 = \stdCone^1\times\stdCone^2
$$
be a chain of faces of $\stdCone^1\times \stdCone^2$  together with $z_i \in \stdFace _i^*\cap\hat \stdSpace^\perp \cap \{\hat a\}^\perp$ such that 
that $(\stdFace _{\ell}, \hat \stdSpace, \hat a)$ satisfies the 
PPS condition and $\stdFace_{i+1} = \stdFace _i\cap \{z_i\}^\perp$ for every $i$.

For $i = 1,\ldots, \ell - 1$, let $\psi _{i}$ be a {\FRF} of $\stdCone^1\times\stdCone^2$ with respect to $\stdFace_{i}$, $z_i$. 
Then, after positive rescaling the $\psi _{i}$, there is a positive constant $\kappa$ (depending on $\stdCone^1,\stdCone^2,\stdSpace,a$) such that whenever  $x \in \spanVec (\stdCone^1\cap \stdCone^2)$ satisfies the inequalities
$$
\quad \dist(x,\stdCone^1\cap \stdCone^2) \leq \epsilon, \quad \dist(x,\stdSpace + a) \leq \epsilon,
$$
we have 
$$
\dist\left(x, (\stdSpace + a) \cap \stdCone^1\cap \stdCone^2 \right) \leq (\kappa \norm{x} + \kappa )(\epsilon+\varphi(\epsilon,\norm{x})),
$$
where $
\varphi = \psi _{{\ell-1}}\comp \cdots \comp \psi_{{1}}$, if $\ell \geq 2$. If $\ell = 1$, we let $\varphi$ be the function satisfying $\varphi(\epsilon, \norm{x}) = \epsilon$.	
\item There exists at least one chain satisfying the requirements in item $(i)$ of length $\dpp(\hat \stdSpace, \hat a )+1$, where $\dpp(\hat \stdSpace, \hat a )$ is the minimum number of facial reduction steps necessary to find a face satisfying the PPS condition for the problem 
$(\stdCone^1\times\stdCone^2,\hat \stdSpace,\hat a)$. The following inequality holds.
\begin{equation*}
\dpp(\hat \stdSpace, \hat a ) \leq \min \{\distP(\stdCone^1)+\distP(\stdCone^2), \dim( \hat \stdSpace^\perp \cap \{a\}^\perp), \ds(\hat\stdSpace,\hat a) \}.
\end{equation*}
\end{enumerate}
\end{proposition}
\begin{proof}
Item $(i)$ is a consequence of applying Theorem \ref{theo:err} to $(\stdCone^1\times \stdCone^2, \hat \stdSpace, \hat a)$ together with \eqref{eq:inter:1}, \eqref{eq:inter:2} and \eqref{eq:inter:3}, rescaling the functions $\psi_{i}$ if necessary. 
Item $(ii)$ is a direct consequence of Proposition \ref{prop:effic}.
\end{proof}
We conclude this subsection with an application of Proposition~\ref{prop:inter}. 
Let $\mathcal{N}^n$ denote the cone of $n\times n$ symmetric matrices with nonnegative entries. Then, the doubly nonnegative cone $\doubly{n}$ is defined as
the intersection $\PSDcone{n}\cap \mathcal{N}^n$. It corresponds to the matrices that 
are simultaneously positive semidefinite and nonnegative.
The cone $\doubly{n}$ has found many applications recently, see \cite{Yoshise10,KKT16,AKKT2014,AKKT17}.
\begin{proposition}[Error bound for the doubly nonnegative cone]\label{prop:doubly}
	
Suppose $(\doubly{n},\stdSpace,a)$ is feasible.
Then, there is a positive 
constant $ \kappa$ (depending on $n,\stdSpace,a$) such that whenever  $x$  satisfies the inequalities
$$
\quad \dist(x,\doubly{n}) \leq \epsilon, \quad \dist(x,\stdSpace + a) \leq \epsilon,
$$
we have 
$$
\dist\left(x, (\stdSpace + a) \cap \doubly{n}\right) \leq 
(\kappa \norm{x} + \kappa)\left(\sum _{j = 0}^{\dpp(\hat \stdSpace, \hat a )} \epsilon^{({2^{-j}}) }\norm{x}^{1-{2^{-j}} }\right),
$$
where $\hat \stdSpace$ and $\hat a$ are as in Proposition \ref{prop:inter}. 
Furthermore, 
$$
\dpp(\hat \stdSpace, \hat a ) \leq \min \left\{n-1,\dim (\hat\stdSpace^\perp \cap \{\hat a\}^\perp), \ds(\hat\stdSpace,\hat a)  \right\}.
$$
\end{proposition}
\begin{proof}
We apply Proposition~\ref{prop:inter} to $\doubly{n} = \PSDcone{n} \cap \mathcal{N}^n$.
From item $(i)$ of Proposition~\ref{prop:ebtp}, 
we know that {\FRFs}  for $\PSDcone{n} \times \mathcal{N}^n$ can be taken to be positive rescalings of the sum of {\FRFs}  for $\PSDcone{n}$ and 
$\mathcal{N}^n$. 
From Theorem \ref{theo:sym_am} and Proposition \ref{prop:poly_amenable}, 
we conclude that  {\FRFs} for $\PSDcone{n} \times \mathcal{N}^n$ can be taken to be
$$
\psi _{i}(\epsilon,\norm{x}) = \kappa _i\epsilon +  \kappa_i \sqrt{\epsilon \norm{x}}. 
$$
Then, we apply Lemma \ref{lem:sym_res} and proceed as in 
the proof of Theorem \ref{theo:sym_err}.

The bound on $\dpp(\hat \stdSpace, \hat a )$ follows from item $(ii)$ of 
Proposition \ref{prop:inter} and the fact 
the $\distP(\PSDcone{n}) \leq n-1$ (see Remark~\ref{rem:sym}) and $\distP(\mathcal{N}^n) = 0$, since $\mathcal{N}^n$ is a polyhedral cone.
\end{proof}
We can also prove a result similar to Proposition \ref{prop:sym_err2}. For example, if  we impose $\epsilon \leq 1$ and $\norm{x} \leq \rho$ and  use the fact that $\dpp(\hat \stdSpace, \hat a ) \leq n-1$, we may adjust the constant $\kappa$ so that the bound becomes
\begin{equation}\label{eq:dnn}
\dist\left(x, (\stdSpace + a) \cap \doubly{n}\right) \leq \kappa \epsilon^{({2^{1-n}})},
\end{equation}
where, analogous to Proposition~\ref{prop:sym_err2}, $\kappa$ depends on $\doubly{n}, \stdSpace, a, \rho$.

\begin{remark}
In Example~2 of \cite{ST00}, Sturm constructed a subspace $\stdSpace \subseteq \S^n$ and a sequence of matrices $\{x_{\epsilon}\mid \epsilon >0  \}$ contained in $\PSDcone{n}$, such that $\dist(x_{\epsilon},\stdSpace) < \epsilon$ but $\dist(x_{\epsilon},\PSDcone{n}\cap \stdSpace) \geq \epsilon^{1/2^{n-1}}$, for every $\epsilon$. 
This happens because all matrices in $\PSDcone{n}\cap \stdSpace$ are such that their $(1,n)$ entry is $0$, while the $(1,n)$ entry of $x_{\epsilon}$ is $\epsilon^{1/2^{n-1}}$.

Therefore, apart from the constant $\kappa$, Sturm's error bound can be tight. However, a closer inspection shows that the $x_{\epsilon}$ are, in fact, doubly nonnegative matrices and the same reasoning shows that $\dist(x_{\epsilon},\doubly{n}\cap \stdSpace) \geq \epsilon^{1/2^{n-1}}$. It follows that \eqref{eq:dnn} is also tight in the same sense.
\end{remark}

\section{Conclusion and open problems}\label{sec:conc}
In this paper, we introduced the concepts of \emph{amenable cones} and \emph{\FRFs} (\FRFa s), which makes it possible to derive error bound 
results for problems that do not satisfy regularity conditions. 
As applications, we gave H\"olderian error bounds for symmetric cones (Theorem~\ref{theo:sym_err} and Proposition~\ref{prop:sym_err2}) and 
for doubly nonnegative cones (Proposition~\ref{prop:doubly}). 
In summary, given some pointed closed convex cone $\stdCone$, we need the following ingredients for obtaining error bounds under our approach:
\begin{enumerate}
	\item First, it is necessary to prove that $\stdCone$ is \emph{amenable} (Definition \ref{def:beaut}). 
	\item Then, we must work out the {\FRFs} as in 
	Theorem \ref{theo:sym_am}.
	\item Finally, we apply Theorem \ref{theo:err} and try to 
	obtain an upper bound for the composition of {\FRFs} as in Lemma \ref{lem:sym_res}. 
	We can further restrict $\epsilon$ (as in Propositions~\ref{prop:sym_err2}) and/or change the distance functions (Remark~\ref{rem:dist}).
\end{enumerate} 
If it is hard to prove that $\stdCone$ is amenable, we might be able to express $\stdCone$ as an intersection of amenable cones and apply Proposition \ref{prop:inter}.

We will now point out some open questions and directions 
for future work.
\begin{enumerate}
	\item \emph{Which cones are amenable? Do they admit simple {\FRFs}?}
We proved that symmetric cones are amenable, but how about homogeneous cones? An answer to this question might follow from the fact that homogeneous cones are ``slices of positive semidefinite cones'', see the work by Chua \cite{CH03} and Proposition 1 together with section 4 of the work by Faybusovich \cite{FB02}. 

Another interesting family of cones to investigate are the 
$p$-cones. They are defined as $$\SOC{p}{n}  = \{(t,x) \in \Re\times \Re^{n-1} \mid t \geq \norm{x}_p \},$$
where $\norm{\cdot}_p$ denotes the $p$-norm. 
For $p = 1$ or $p = \infty$ or $n < 3$, $\SOC{p}{n}$ is polyhedral and hence it is amenable by Proposition \ref{prop:beaut}.
If $p = 2$, then $\SOC{p}{n} $ becomes the second order cone which is a symmetric cone, so Theorem \ref{theo:sym_am} applies.
It remains to analyze the case $1 < p < \infty$, $p \neq 2$ and $n \geq 3$. For those $p$, since $\SOC{p}{n}$ is strictly convex, we know from Proposition~\ref{prop:beaut} that $\SOC{p}{n}$ is  amenable. 
However, we do not know how obtain {\FRFa}s that are 
simpler than the canonical one. We remark that it was recently shown that $\SOC{p}{n}$ is not homogeneous for those $p$, see the work by Ito and Louren\c{c}o \cite{IL17_2}. Therefore, computing 
{\FRFa}s for homogeneous cones will not be helpful here. 
We conjecture that  $\SOC{p}{n}$ admit {\FRFa}s  of the form $\kappa \epsilon + \kappa(\norm{x} \epsilon)^{1/p}$.

\item It might be possible to relax Definition \ref{def:beaut} and obtain error bound results for cones that are not
amenable. For example, we could require that for every face $\stdFace$ of $\stdCone$, there should be some $\kappa > 0$ and 
$\gamma \in (0,1]$ such that   
$$
\dist(x, \stdFace) \leq \kappa \dist(x,\stdCone)^\gamma,
$$
for every $x \in \spanVec \stdFace$ satisfying $\dist(x,\stdCone) \leq 1$.
For cones satisfying this property,  it seems that a result similar to 
Proposition \ref{prop:err2} might hold. 
In this case, it could be possible to obtain a result analogous to Theorem \ref{theo:err}.

\end{enumerate}
\appendix
\section{Miscellaneous proofs}\label{app:proof}
\subsubsection*{Proof of Proposition \ref{prop:beaut_prev}}
		\begin{enumerate}[label=$(\roman*)$]	
			\item Let $\stdFace$ be a face of $\stdCone^1\times \stdCone^2$. We have $\stdFace = \stdFace ^1 \times \stdFace^2$, where $\stdFace^1$ and $\stdFace^2$ are faces of $\stdCone^1$ and $\stdCone^2$ respectively. 
			From our assumptions in Section \ref{sec:prel}, Equation~\eqref{eq:prod_proj} and the amenability of $\stdCone^1$ and $\stdCone^2$, it follows that  there are positive constants $\kappa_1, \kappa _2$ such that 
			$$
			\dist((x_1,x_2),\stdFace) \leq \sqrt{\kappa_1^2\dist(x_1,\stdCone^1)^2 + 
				\kappa _2^2\dist(x_2,\stdCone^2)^2},
			$$
			whenever $x_1 \in \spanVec \stdFace^1$ and $x_2 \in \spanVec \stdFace^2$. Therefore,
			\begin{align*}
			\dist((x_1,x_2),\stdFace) &\leq \max\{\kappa_1,\kappa_2\}\sqrt{\dist(x_1,\stdCone^1)^2 + \dist(x_2,\stdCone^2)^2}\\
			& = \max\{\kappa_1,\kappa_2\}\dist((x_1,x_2), \stdCone^1\times \stdCone^2),
			\end{align*}
			whenever $(x_1,x_2) \in \spanVec (\stdFace^1\times \stdFace^2) = (\spanVec \stdFace^1) \times (\spanVec \stdFace^2)$. 
			\item If $\stdMap$ is the zero map, we are done, since 
			$\{0\}$ is amenable. So, suppose that $\stdMap$ is a nonzero injective linear map.
			Then, the faces of $\stdMap(\stdCone)$ are images of faces of $\stdCone$ by $\stdMap$. 
			Accordingly, let $\stdFace \face \stdCone$.
			Because $\stdCone$ is amenable, there is $\kappa$ such that 
			\begin{align}
			\dist(x,\stdFace) \leq \kappa \dist(x, \stdCone), \quad \forall x \in \spanVec \stdFace.  \label{eq:beaut:im}
			\end{align}
			
			As $\stdMap$ is a linear map, we have $\spanVec \stdMap(\stdFace) = \stdMap(\spanVec \stdFace)$.  
			Let $\sigma _{\min}, \sigma _{\max}$ denote, respectively, the minimum and maximum singular values of $\stdMap$. We have
			\begin{align*}
			\sigma _{\min} = \min \{\norm{Ax} \mid \norm{x} = 1 \}, \quad 
			\sigma _{\max} = \max \{\norm{Ax} \mid \norm{x} = 1 \}.
			\end{align*}
			They are both positive since $\stdMap$ is injective and nonzero. 
			Now, let  $x \in \spanVec \stdFace$, then
			\begin{align*}
			\dist(\stdMap(x),\stdMap(\stdFace)) & \leq \sigma _{\max }\dist(x,\stdFace)&\\
			& \leq {\kappa}{\sigma _{\max}}\dist(x,\stdCone)& \text{(From \eqref{eq:beaut:im})}\\
			& \leq \frac{{\kappa}{\sigma _{\max}}}{\sigma _{\min}}\dist(\stdMap(x),\stdMap(\stdCone)). &
			\end{align*}
		\end{enumerate}	
\subsubsection*{Proof of Proposition \ref{prop:am_others}}
\begin{proof}
	\fbox{$(i) \Rightarrow (ii)$}
	Let $x,u \in \jAlg$ be such that $x+u\in \spanVec \stdFace$ and $\norm{u} = \dist(x,\spanVec \stdFace)$. Since $\dist(\cdot,\stdCone)$ and 
	$\dist(\cdot,\stdFace)$ are sublinear functions, we have that \eqref{eq:def_am} implies that
	\begin{align}
	\dist(x,\stdFace) &\leq \dist(-u,\stdFace) + \dist(x+u,\stdFace) \notag \\
	& \leq \dist(x,\spanVec \stdFace) + \kappa(\dist(x+u,\stdCone)) \notag\\
	& \leq \dist(x,\spanVec \stdFace) + \kappa(\dist(x,\stdCone)+\dist(x,\spanVec \stdFace))\notag \\
	& \leq (1+\kappa)(\dist(x,\stdCone)+\dist(x,\spanVec \stdFace)), \quad \forall x \in \jAlg. \label{eq:subtr}
	\end{align}
	Here we used the fact that $\dist(-u,\stdFace) \leq \norm{-u}$, since $0 \in \stdFace$.
	This shows that $(i) \Rightarrow (ii)$.
	
	\fbox{$(ii) \Rightarrow (i)$}	Since
	$\stdCone$ and $\spanVec \stdFace$ intersect at $0$ subtransversally, there is $\delta > 0$ such that
	$$
	\dist(z,\stdFace) \leq \kappa(\dist(z,\stdCone)+\dist(z,\spanVec \stdFace)), \quad \forall z \text{ with } \norm{z} \leq \delta.
	$$
	Therefore, if $x \in \jAlg$ is nonzero, we have
	$$
	\dist(\delta \frac{x}{\norm{x}},\stdFace) \leq \kappa(\dist(\delta \frac{x}{\norm{x}},\stdCone)+\dist(\delta \frac{x}{\norm{x}},\spanVec \stdFace)).
	$$
	Now, we recall that if $C$ is a convex cone, then $\dist(\alpha x, C) = \alpha \dist(x,C)$ for every positive $\alpha$. We conclude that 
	$$
	\dist(x,\stdFace) \leq \kappa(\dist(x,\stdCone)+\dist(x,\spanVec \stdFace)), \quad \forall x \in \jAlg.
	$$%
	Therefore, if $x \in \spanVec \stdFace$, then 
	$\dist(x, \stdFace) \leq \kappa \dist(x,\stdCone)$.
	
	\fbox{$(i) \Rightarrow (iii)$} The inequality in \eqref{eq:subtr} shows that 
	$$\dist(x,\stdFace) \leq (2+2\kappa)\max(\dist(x,\stdCone), \dist(x,\spanVec \stdFace)), \quad \forall x \in \jAlg.	
	$$
	Therefore, $\stdCone$ and $\spanVec \stdFace$ are boundedly linearly regular.
	
	\fbox{$(iii) \Rightarrow (ii)$} Let $U = \{x \in \jAlg \mid \norm{x} \leq 1 \}$. Then, there exists $\kappa _U$ such that 
	\begin{align*}
	\dist(x,\stdFace) &\leq \kappa _U\max(\dist(x,\stdCone), \dist(x,\spanVec \stdFace)) \\ 
	& \leq \kappa_U (\dist(x,\stdCone)+ \dist(x,\spanVec \stdFace)), \quad \forall x \in U.	
	\end{align*}
	Therefore, $\stdCone$ and $\spanVec \stdFace$ intersect subtransversally at $0$. 
\end{proof}

\subsubsection*{Proof of Proposition \ref{prop:ebtp}}
	\begin{enumerate}
			\item 
			Suppose that $x \in \spanVec \stdCone$ satisfies  the inequalities
			\begin{equation}\label{eq:prod}
			\dist(x,\stdCone) \leq \epsilon, \quad \inProd{x}{z} \leq \epsilon, \quad \dist(x, \spanVec \stdFace ) \leq \epsilon.
			\end{equation}
			
			Note that
			$$
			\stdFace\cap \{z\}^{\perp} = (\stdFace^{1} \cap\{z_1\}^\perp) \times (\stdFace^{2} \cap\{z_2\}^\perp). 
			$$
			Also, due to our assumptions (Section \ref{sec:conv}), we have
			$$\norm{x-y}^2 = \norm{x_1-y_1}^2 + \norm{x_2-y_2}^2$$ for every 
			$x,y \in \jAlg^1\times \jAlg^2$.
			Thus we have the following implications:
			\begin{align}
			\dist(x,\stdCone) \leq \epsilon \quad &  \Rightarrow\quad \dist(x_1,\stdCone^1) \leq \epsilon, \quad \dist(x_2,\stdCone^2) \leq \epsilon \label{eq:im1}\\
			\dist(x,\spanVec \stdFace ) \leq \epsilon\quad &\Rightarrow\quad  \dist(x_1,\spanVec \stdFace^1) \leq \epsilon, \quad  \dist(x_2,\spanVec \stdFace^2) \leq \epsilon \label{eq:im2}
			\end{align}
			The first step is showing that there are positive constants $\kappa _1$ and $\kappa _2$ such that for all $x \in \jAlg^1\times\jAlg^2$, we also have
			\begin{align}
			x \text{ satisfies \eqref{eq:prod}} &  \Rightarrow\quad \inProd{x_1}{z_1} \leq \kappa_1\epsilon\quad \text{and} \quad
			\inProd{x_2}{z_2} \leq \kappa _2\epsilon.\label{eq:x_prod}
			\end{align}
			Suppose $x$ satisfies \eqref{eq:prod}.		
			By \eqref{eq:im2}, we have $\dist(x_1, \spanVec{\stdFace^1}) \leq \epsilon$. 
			Therefore, there exists $y_1 \in \jAlg^1$ such that $x_1 + y_1 \in \spanVec{\stdFace^1} $ and $\norm{y_1} \leq \epsilon$. 
			Due to \eqref{eq:im1} and the amenability of $\stdCone^1$, there exists $\hat \kappa_1$ (not depending on $x_1$) such that
			\begin{equation*}
			\dist(x_1 + y_1, \stdFace^1) \leq \hat \kappa_1 \dist(x_1 + y_1, \stdCone^1) \leq 2\epsilon \hat \kappa_1. 
			\end{equation*}
			Therefore, there exists $v_1 \in \jAlg^1$ such that $\norm{v_1} \leq 2\epsilon \hat \kappa_1$ and $$x_1 + y_1 + v_1 \in \stdFace^1.$$
			In a completely analogous manner, there is a constant $\hat \kappa_2> 0$ and there are $y_2,v_2 \in \jAlg^2$
			such  $$x_2 + y_2 + v_2 \in \stdFace^2,$$ with $\norm{y_2} \leq \epsilon$ and $\norm{v_2} \leq 2\epsilon \hat \kappa_2$. It follows 
			that 
			\begin{align*}
			\inProd{(x_1+y_1+v_1,x_2+y_2+v_2)}{(z_1,z_2)} \leq M\epsilon,
			\end{align*}
			for $M = 1 + \norm{z_1} + 2\hat \kappa_1 + \norm{z_2} + 2\hat \kappa_2$. Since $\inProd{x_1+y_1+v_1}{z_1} \geq 0$ and 
			$\inProd{x_2+y_2+v_2}{z_2} \geq 0$, we get 
			\begin{align*}
			\inProd{x_i+y_i+v_i}{z_i} \leq M\epsilon,		
			\end{align*}
			for $i = 1,2$. We then conclude that 
			\begin{align}
			\inProd{x}{z} \leq \epsilon \quad \Rightarrow \quad \inProd{x_i}{z_i} \leq \kappa_i\epsilon,\label{eq:xz}	
			\end{align}
			whenever $x$ satisfies \eqref{eq:im1} and \eqref{eq:im2}, where $\kappa _i = M+ \norm{z_i} + \norm{z_i}2\hat \kappa_i$.
			
			Now, let $\psi_{\stdFace_1,z_1}$ and $\psi_{\stdFace_2,z_2}$ be arbitrary {\FRFs} for 
			$\stdFace_1,z_1$ and $\stdFace_2,z_2$, respectively.
			We positive rescale $\psi_{\stdFace_1,z_1}$ and 
			$\psi_{\stdFace_2,z_2}$ so that 
			\begin{align*}
			\dist(x_i,\stdCone) \leq \epsilon,\quad \inProd{x_i}{z_i} \leq \kappa_i\epsilon, \quad \dist(x, \spanVec \stdFace_i ) \leq \epsilon  
			\end{align*}
			implies $\dist(x_i,\hat \stdFace_i) \leq  \psi_{\stdFace_i,z_i}(\epsilon,\norm{x_i})$, for $i = 1,2$. 
			
			Finally, from \eqref{eq:im1}, \eqref{eq:im2}, \eqref{eq:xz} and using the fact that $\psi_{\stdFace_1,z_1}$ and $\psi_{\stdFace_2,z_2}$ are monotone nondecreasing on the second argument we conclude that whenever $x$ satisfies \eqref{eq:prod} we have
			\begin{align}
			\dist(x,\hat \stdFace)& = 
			\sqrt{\dist(x_1,\hat \stdFace^1)^2 + \dist(x_2,\hat \stdFace^2)^2  }\notag \\
			&\leq {\dist(x_1,\hat \stdFace^1)} +  {\dist(x_2,\hat \stdFace^2)} \notag \\
			&\leq \psi_{\stdFace_1,z_1}(\epsilon,\norm{x}) + \psi_{\stdFace_2,z_2}(\epsilon,\norm{x}). \label{eq:res_sum} \notag
			\end{align}
			Therefore, $\psi_{\stdFace_1,z_1}+\psi_{\stdFace_2,z_2}$ is a {\FRF} for $\stdFace,z$. 
			
			\item
			%
			The proposition is true if $\stdMap$ is the zero map, so suppose that $\stdMap$ is a nonzero injective linear map.
			First, we observe that 
			$$
			(\stdMap(\stdFace))\cap \{z\}^\perp = \stdMap(\stdFace\cap\{\stdMap^\T z \}^\perp ).
			$$

			Let $\hat \stdFace = \stdFace\cap\{\stdMap^\T z \}^\perp$.
			Let $\psi _{\stdFace,\stdMap^\T z}$ be a 
			{\FRF} for $\stdFace$ and $\stdMap^\T z$. 
			Let $\sigma _{\min}$ denote the minimum singular value of
			$\stdMap$. We note that $\sigma _{\min}$ is positive because $\stdMap$ is injective.
			We 
			positive rescale $\psi _{\stdFace,\stdMap^\T z}$
			so that whenever $x$ satisfies
			$$
			\dist(x,\stdCone) \leq \frac{1}{\sigma_{\min}}\epsilon, \quad \inProd{x}{\stdMap^\T z} \leq \epsilon, \quad \dist(x, \spanVec \stdFace ) \leq \frac{1}{\sigma_{\min}}\epsilon 
			$$
			we have:
			$$
			\dist(x, \hat \stdFace)  \leq 
			\psi _{\stdFace,\stdMap^\T z} (\epsilon, \norm{x}).
			$$
			Then, we have the following implications:
			\begin{align*}
			\dist(\stdMap(x),\stdMap(\stdCone)) \leq \epsilon \quad & \Rightarrow \quad \dist(x,\stdCone) \leq \frac{1}{\sigma_{\min}}\epsilon \\
			\inProd{\stdMap(x)}{z} \leq \epsilon \quad & \Leftrightarrow \quad \inProd{x}{\stdMap^\T z} \leq \epsilon\\
			\dist(\stdMap(x),\spanVec \stdMap(\stdFace)) \leq \epsilon \quad & \Rightarrow \quad{\dist(x,\spanVec \stdFace) \leq \frac{1}{\sigma_{\min}}\epsilon}\\	
			\dist(\stdMap(x), \stdMap(\hat \stdFace)) \leq \sigma _{\max}\psi _{\stdFace,\stdMap^\T z} (\epsilon, \norm{\stdMap x}/\sigma_{\min})  \quad &  \Leftarrow \quad \dist(x, \hat \stdFace)  \leq 
			\psi _{\stdFace,\stdMap^\T z} (\epsilon, \norm{x}),
			\end{align*}
			where $\sigma _{\max}$ is the maximum singular value of $\stdMap$. This shows that we can use $$\tilde \psi _{\stdMap(\stdFace),z}(\epsilon, \norm{\stdMap x}) = \sigma _{\max}\psi (\epsilon, \norm{\stdMap x}/\sigma _{\min})$$ as a {\FRF} for $\stdMap(\stdFace)$ and $z$.

		\end{enumerate}

\small{
	\section*{Acknowledgements}
	We thank the editors and four referees for their  insightful comments, which helped to improve the paper substantially. In particular, the discussion on tangentially exposed cones and subtransversality was suggested by Referee~1. Also, comments from Referees~1 and 2 motivated Remark~\ref{rem:const}.
	We would like to thank Prof.~G\'abor Pataki for helpful advice and for suggesting that we take a look at projectionally exposed cones. Incidentally, this was also suggested by Referee~4. Referee~4 also suggested the remark on the tightness of the error bound for doubly nonnegative matrices.	
	Feedback and encouragement from Prof.~Tomonari Kitahara, Prof.~Masakazu Muramatsu and Prof.~Takashi Tsuchiya were  highly helpful and
	they  provided the official translation of ``amenable cone'' to Japanese: \includegraphics[scale=0.8]{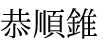} (\textit{kyoujunsui}).
	This work was partially supported by the Grant-in-Aid for Scientific Research (B) (18H03206) and  the Grant-in-Aid for Young Scientists (19K20217) from Japan Society for the Promotion of Science.
}
\bibliographystyle{abbrvurl}
\bibliography{journal-titles,references,bib_plain}
\end{document}